\documentclass{aic}

\usepackage[utf8]{inputenc}

\usepackage{amsthm}
\newtheorem{theorem}{Theorem}
\newtheorem{corollary}[theorem]{Corollary}
\newtheorem{lemma}[theorem]{Lemma}
\newtheorem*{klaim}{Claim}
\newtheorem{conjecture}[theorem]{Conjecture}
\newtheorem{question}[theorem]{Question}

\usepackage{algorithm}
\usepackage[noend]{algpseudocode}
\usepackage{enumitem}

\DeclareMathOperator\tw{tw}
\DeclareMathOperator\stw{stw}
\DeclareMathOperator\lev{lev}
\DeclareMathOperator\lin{lin}
\DeclareMathOperator\modu{mod}
\DeclareMathOperator\clos{clos}

\let\epsi\varepsilon
\def\email#1{\texttt{\small{#1}}}
\newcommand{\set}[1]{\{#1\}}

\newcommand{\Oh}{\mathcal{O}} 
\newcommand{\calC}{\mathcal{C}}
\newcommand{\calD}{\mathcal{D}}
\newcommand{\calF}{\mathcal{F}}
\newcommand{\calB}{\mathcal{B}}

\aicAUTHORdetails{%
 title  = {Improved Bounds for Centered Colorings}
 author = {Michał Dębski, Stefan Felsner, Piotr Micek, and Felix Schröder},
 plaintextauthor = {Michal Debski, Stefan Felsner, Piotr Micek, and Felix Schroeder}
 plaintexttitle = {Improved Bounds for Centered Colorings}
 runningtitle = {Centered Colorings}
 runningauthor = {M.\ Dębski, S.\ Felsner, P.\ Micek, and F.\ Schröder}
 keywords = {bounded expansion, topological minor, planar, structure theorem, entropy compression}
}
\aicEDITORdetails{%
   year={2021},
   number={8},
   received={25 August 2020},   
   published={16 August 2021},  
   doi={10.19086/aic.27351},      
}   
\begin{document}

\begin{frontmatter}[classification=text]

\title{Improved Bounds for Centered Colorings}

\author[MD]{Michał Dębski}
\author[SF]{Stefan Felsner\thanks{Supported by DFG grant FE 340/13-1.}}
\author[PM]{Piotr Micek\thanks{Supported partially by a Polish National Science Center grant
(BEETHOVEN; UMO-2018/31/G/ST1/03718).}}
\author[FS]{Felix Schröder}

\renewcommand{\thefootnote}{\fnsymbol{footnote}}
\footnotetext{A preliminary version of this paper appeared as an extended
  abstract in the Proceedings of the 31st annual ACM-SIAM Symposium on Discrete Algorithms, SODA 2020~\cite{DFMS20}.}

\begin{abstract}
  A vertex coloring $\phi$ of a graph $G$ is $p$-\emph{centered} if for every
  connected subgraph~$H$ of $G$ either $\phi$ uses more than $p$ colors on $H$
  or there is a color that appears exactly once on~$H$.  Centered colorings
  form one of the families of parameters that allow to capture notions of
  sparsity of graphs: A class of graphs has bounded expansion if and only if
  there is a function~$f$ such that for every $p\geq1$, every graph in the
  class admits a $p$-centered coloring using at most $f(p)$ colors.

  In this paper, we give upper bounds for the maximum number of colors needed
  in a $p$-centered coloring of graphs from several widely studied graph
  classes.  We show that: (1)~planar graphs admit $p$-centered colorings with
  $\Oh(p^3\log p)$ colors where the previous bound was $\Oh(p^{19})$; (2)
  bounded degree graphs admit $p$-centered colorings with $\Oh(p)$ colors
  while it was conjectured that they require an exponential number of colors
  in~$p$; (3) graphs avoiding a fixed graph as a topological minor admit
  $p$-centered colorings with a polynomial in $p$ number of colors.  All these
  upper bounds imply polynomial algorithms for computing the colorings.

  Prior to this work there were no non-trivial lower bounds known.  We show
  that: (4) there are graphs of treewidth $t$ that require~$\binom{p+t}{t}$
  colors in any $p$-centered coloring; this matches the known upper bound.
  (5)~there are planar graphs that require $\Omega(p^2\log p)$ colors in any
  $p$-centered coloring.  We also give asymptotically tight bounds for
  outerplanar graphs and planar graphs of treewidth $3$.  We prove our results
  using a variety of techniques.  The upper bound for planar graphs involves
  an application of a recent structure theorem while the upper bound for
  bounded degree graphs comes from the entropy compression method.  We lift
  the result for bounded degree graphs to graphs avoiding a fixed topological
  minor using the Grohe\hspace{0.5pt}--Marx structure theorem.
\end{abstract}
\end{frontmatter}

\section{Introduction}
Structural graph theory has expanded beyond the study of
classes of graphs that exclude a fixed minor.
One of the driving forces was, and is, to develop efficient algorithms 
for computationally hard problems for graphs that are `structurally sparse'.
Ne{\v{s}}et{\v{r}}il and Ossona de Mendez introduced the concepts of 
classes of graphs with 
\emph{bounded expansion}~\cite{NOdM08} and 
classes of graphs that are \emph{nowhere dense}~\cite{NOdM11}.
These are very robust properties that include 
every class excluding a fixed minor but also 
graphs of bounded book-thickness or
graphs that allow drawings with bounded number of crossings per edge,
see~\cite{NOdMW12}.
At first sight, bounded expansion might seem to be a weak property for
a graph class.
Yet, this notion captures enough structure to allow solving a wide range
of algorithmic problems efficiently: 
Dvořák, Kráľ and Thomas~\cite{DKT12} devised an FPT algorithm for testing 
first-order definable properties in classes of bounded expansion.

One reason that these new notions attracted much attention 
is the realization that they can be characterized in several, 
seemingly different ways.
Instead of providing the original definition 
we define bounded expansion in terms of centered chromatic numbers.
\subsection{Centered colorings}
Let $p$ be a positive integer.
A vertex coloring $\phi$ of a graph $G$ is $p$-\emph{centered}
if for every connected subgraph $H$ of $G$ 
either $\phi$ uses more than $p$ colors on $H$ or there is a color
that appears exactly once on $H$.
The $p$-\emph{centered chromatic number} $\chi_p(G)$ of $G$ is the minimum integer $k$ 
such that there is a $p$-centered coloring of $G$ using $k$ colors.

A vertex coloring of a graph is $1$-centered if and only if 
it is \emph{proper}, i.e.,  
adjacent vertices receive distinct colors.
A vertex coloring is $2$-centered if and only if it is a \emph{star coloring}, 
i.e., it is proper and every path on four vertices receives at least three colors.
A class $\calC$ of graphs is of \emph{bounded expansion} if and only if 
there exists a function $f:\mathbb{N}\to\mathbb{N}$ such that
for every integer $p\geq1$ and every $G\in\calC$ we have
$\chi_p(G)\leq f(p)$. 
Ne{\v{s}}et{\v{r}}il and Ossona de
Mendez, who originally defined bounded expansion in terms of maximum densities 
of shallow minors, 
showed in~\cite{NOdM08} that the definitions are equivalent.

Colorings, such that the condition depending on $p$ is omitted, are called \emph{centered}
colorings. The \emph{centered chromatic number} is the minimum number of colors in such a
coloring and coincides with the treedepth for every graph $G$~(\cite{NOdM06}, Lemma 4.2).
The \emph{treedepth} of a graph $G$ is the minimum height of a rooted forest 
$F$ such that $G\subseteq \clos(F)$ where 
(1) a \emph{rooted forest} is a
disjoint union of rooted trees; 
(2) the \emph{height} of a rooted forest $F$ is maximum number of vertices on 
a path from a root to a leaf in $F$;
(3) the closure $\clos(F)$ of $F$ is the graph with vertex set $V(F)$ and 
edge set $\set{\set{v,w}\mid \text{$v$ is an ancestor of $w$ in $F$}}$.
As observed in~\cite{NOdM06},  every $p$-centered coloring of a graph $G$ 
is also a \emph{treedepth}-$p$ \emph{coloring} of $G$, i.e., for every $i\leq p$
the union of any $i$ color classes induces a subgraph of $G$ of treedepth at most $i$.

\subsection{Algorithmic applications}
Low treedepth colorings are a central tool for designing parametrized algorithms 
in classes of bounded expansion.
For example, Pilipczuk and Siebertz~\cite{PS19} showed that when $\calC$ 
is a class of graphs avoiding a fixed minor,
then it can be decided whether a given $p$-vertex graph $H$ is a subgraph of a given $n$-vertex
graph $G$ in $\calC$ in time $2^{\Oh(p\log p)}\cdot n^{\Oh(1)}$ and space $n^{\Oh(1)}$.
The algorithm witnessing this statement starts with a computation 
of a $p$-centered coloring of $G$ with $p^{\Oh(1)}$ colors, and 
for each $p$-tuple of colors it applies a procedure 
to solve this problem for graphs $G$ of treedepth at most $p$.
The results of our paper imply a corresponding algorithm for the case where
$\calC$ is a class of graphs avoiding a fixed graph 
as a \emph{topological} minor.

The polynomial space complexity mentioned above is remarkable.  Typically
dynamic programming algorithms on tree decompositions use space exponential in
the width of the decomposition and there are complexity-theoretical reasons to
believe that without significant increase in time complexity, this cannot be
avoided.  On the other hand treedepth decompositions, sometimes called
elimination trees, allow to devise algorithms using only polynomial space in
the height of the decomposition, see a thorough study of this phenomenon by
Pilipczuk and Wrochna~\cite{PW18}.

Clearly, the running times of algorithms based on $p$-centered colorings 
heavily depend on the number of colors used.
A recent experimental work by O’Brien and Sullivan~\cite{BS17}
points to the lack of efficient coloring procedures
as the major bottleneck for applicability of these algorithms in practice.

\subsection{Results} In this paper, we improve the bounds for the maximum
$\chi_p(G)$ for graphs $G$ in several important sparse classes of graphs.
Most importantly, we reduce the upper bound on the number of colors needed for
a $p$-centered coloring of planar graphs from $\Oh(p^{19})$ to
$\Oh(p^3\log p)$ and we show how to color graphs of bounded degree with
$\Oh(p)$ colors; it was previously believed that there was an exponential
lower bound for this latter class.  All our bounds are supported with
polynomial time algorithms computing the coloring.

We proceed with a presentation of our results.

\begin{theorem}
Planar graphs admit $p$-centered colorings with $\Oh(p^3\log p)$ colors. 
\label{thm:planarup}
\end{theorem}
The previously best known bound was $\Oh(p^{19})$ due to
Pilipczuk and Siebertz~\cite{PS19}.
A key tool responsible for the improvement of the exponent is 
a recent structure theorem for planar graphs due to 
Dujmović et al.~\cite{DJMMUW-queues} which has its roots in~\cite{PS19}.
In Section~\ref{sec-application-of-the-structure-thm}, 
we give a precise statement of the theorem and we show how to use it to
color a planar graph in a $p$-centered way with $\Oh(p)\cdot f(p)$ colors 
where~$f(p)$ is the maximum number of colors needed in a $p$-centered coloring
of planar graphs of treewidth at most $3$.
The $\Oh(p^3\log p)$ bound then follows from
Theorem~\ref{thm:pb+p-bounds}.\ref{enu-stackedtri-up}.

Next we conclude an improved bound for graphs drawn on surfaces with bounded genus.
\begin{theorem}\label{thm:genusup}
  Graphs with Euler genus $g$ admit $p$-centered
  colorings with $\Oh(gp+p^3\log p)$ colors.
\end{theorem}
The previous best known bound was $\Oh(g^2p^3+p^{19})$, see~\cite{PS19}.
Our result follows from the bound for planar graphs combined with 
a product structure theorem from~\cite{DJMMUW-queues}.

Graphs with bounded maximum degree are sparse but somehow 
much less structured than planar graphs, e.g. 
they allow any graph as a minor.
It was conjectured in~\cite{PS19}, that the number of colors
needed for $p$-centered colorings in the class of graphs of maximum degree $3$ 
is exponential in $p$.
This was supported by similar bounds for weak coloring numbers 
which is another family of parameters capturing the notion 
of bounded expansion. We disprove the conjecture by providing an upper
bound which is linear in~$p$.
\begin{theorem}\label{thm:bounded-degree}
Graphs with bounded degree admit $p$-centered colorings with $\Oh(p)$ colors.
More specifically, for every $p\geq1$ 
and every graph $G$ with maximum degree at most $\Delta$:
\[
\chi_p(G)\in \Oh\left(\Delta^{2-1/p} \cdot p\right).
\]
\end{theorem}

Dubois et al.~\cite{DJPPP20} recently proved an
almost matching lower bound. By analyzing an appropriate random graph
they show that there are graphs of maximum degree $\Delta$ that
require $\Omega(\Delta^{2-1/p}\cdot p\cdot \ln^{-1/p}\Delta)$ colors in any
$p$-centered coloring.

We prove Theorem~\ref{thm:bounded-degree} with an entropy compression type
argument.  The idea is to run a naïve randomized algorithm for coloring a
graph in the $p$-centered way.  Assuming that this algorithm fails on a long
run, for every possible evaluation of random experiments, we can compress a
random string of bits below the entropy bound. The contradiction shows that
there must be a sequence of bits which leads to a successful run.  The entropy
compression method is inspired by the algorithmic proof of the Lovász local
lemma by Moser and Tardos~\cite{MT10}.  An instructive overview of the method
can be found at Tao's blog~\cite{Tao09} or in a paper by Grytczuk, Kozik, and
Micek~\cite{GKM13} where nonrepetitive sequences over a finite alphabet are
constructed.  In the case $p=2$, i.e., for star chromatic number, our argument
matches the upper bound~$\Oh(\Delta^{1.5})$ by Fertin, Raspaud, and
Reed~\cite{FRR04}.

One of the strong results from~\cite{PS19} is that graphs excluding a fixed minor 
admit $p$-centered colorings with a polynomial in $p$ number of colors. 
The proof goes through the graph minor structure theorem by Robertson and Seymour~\cite{RS03}.
Grohe and Marx~\cite{GM15} extended the structure theorem to graphs avoiding a fixed topological minor.
In some sense, they showed that incorporating the bounded-degree graphs into the structural decomposition 
given in~\cite{RS03} is enough to obtain an approximate characterization for 
classes excluding a fixed topological minor.
For these reasons, we have been able to lift our result for bounded degree graphs and 
obtain the following general statement.

\begin{theorem}\label{thm:topological-minors-sugar}
For every graph $H$ there is a polynomial $f$ such that 
the graphs excluding~$H$ as a topological minor admit $p$-centered
colorings with at most $f(p)$ colors.
\end{theorem}

Prior to our work there were no non-trivial lower bounds for the maximum number of 
colors required in a $p$-centered coloring.
The lower bounds for planar graphs, graphs with bounded treewidth and even 
graphs excluding a fixed minor were only linear in $p$.
We present constructions forcing $\chi_p(G)$ to be superlinear in $p$.
\begin{theorem}\label{thm:tw-lower-bound}
For every $p\geq0$ and $t\geq 0$, there is a graph $G$ of treewidth at most $t$ with
\[
\chi_p(G) \geq \textstyle \binom{p+t}{t}.
\]
\end{theorem}
This lower bound is sharp as Pilipczuk and Siebertz~\cite{PS19} showed 
that for every $p\geq1$, every $t\geq1$,
every graph $G$ of treewidth at most $t$ has $\chi_p(G)\leq \binom{p+t}{t}$.
In fact, we show a slightly stronger statement than Theorem~\ref{thm:tw-lower-bound} 
concerning a relaxation of $p$-centered colorings, 
namely the $p$-linear colorings introduced in~\cite{KBPS18}.
See the precise statement in Section~\ref{sec-lower-bounds-bounded-tw-linear}.

Finally, we present asymptotically tight lower and upper bounds for
the maximum $p$-centered chromatic numbers of outerplanar graphs 
and planar graphs of treewidth at most $3$, i.e., subgraphs of stacked triangulations (see \cite{KV12} for more details).
In particular Theorem \ref{thm:pb+p-bounds}.\ref{enu-stackedtri-low} implies the best known lower bound
$\Omega(p^2\log p)$ for the maximum $p$-centered chromatic number
of planar graphs.

\begin{theorem}\label{thm:pb+p-bounds}\hfill
\begin{enumerate}[label=\textup{(\roman*)},nosep]
\item Outerplanar graphs admit $p$-centered colorings with $\Oh(p\log p)$ colors.
\label{enu-outerplanar-up}
\item There is a family of outerplanar graphs that requires $\Omega(p\log p)$ colors in any $p$-centered coloring.
\label{enu-outerplanar-low}
\item Planar graphs of treewidth at most $3$, 
i.e., subgraphs of stacked triangulations, 
admit $p$-centered colorings with $\Oh(p^2\log p)$ colors.
\label{enu-stackedtri-up}
\item There is a family of planar graphs of treewidth $3$ that requires $\Omega(p^2\log p)$ colors in any $p$-centered coloring.
\label{enu-stackedtri-low}
\end{enumerate}
\label{thm:planarclasses}
\end{theorem}

It turns out that the lower bounds~\ref{enu-outerplanar-low} and~\ref{enu-stackedtri-low}
and the upper bounds ~\ref{enu-outerplanar-up} and~\ref{enu-stackedtri-up} of Theorem~\ref{thm:pb+p-bounds} are initial steps of a more general result stated in Theorem~\ref{thm:simpletree}.
The statement of the theorem involves the notion of
\emph{simple treewidth} which will be formally introduced in
Subsection~\ref{subsec:simpletree-upper}, for the moment we just remark that
graphs of simple treewidth at most 2 are outerplanar graphs and graphs of
simple treewidth at most 3 are planar graphs of treewidth at most 3.

\begin{theorem}\label{thm:simpletree}\hfill
  \begin{enumerate}[label=\textup{(\roman*)},nosep]
  \item \label{enu-simpletree-up}
    Graphs of simple treewidth at most $k$ admit $p$-centered colorings
    with $\Oh(p^{k-1}\log p)$ colors.
  \item \label{enu-simpletree-low}
    There is a family of graphs of simple treewidth $k$ that requires
    $\Omega(p^{k-1}\log p)$ colors in any $p$-centered coloring.
  \end{enumerate} 
\end{theorem}

\subsection{Paper Overview}  In
Section~\ref{sec-application-of-the-structure-thm}, we present a new structure
theorem for planar graphs from~\cite{DJMMUW-queues} and show how to apply it
to get good bounds for $\chi_p(G)$ when $G$ is planar or when $G$ has bounded
genus.  In Section~\ref{sec-bounded-degree}, we set up an entropy compression
argument to prove an $\Oh(p)$ upper bound for graphs with bounded degree.  In
Section~\ref{sec-lifting}, we lift the previous result to obtain a polynomial
bound for graphs avoiding a fixed graph as a topological minor.  In
Sections~\ref{sec-upper-bounds-subplanars} and
\ref{sec-lower-bounds-subplanars}, we give the upper bounds and lower bounds,
respectively, for graphs of bounded simple treewidth.
Section~\ref{sec-lower-bounds-bounded-tw-linear} is devoted to the lower bound
for graphs of bounded treewidth.

Following the presentation of each upper bound, we give a short justification 
why the proof gives a polynomial time algorithm computing the actual coloring.

\section{Upper bounds for planar graphs: An application of the product
structure theorem}
\label{sec-product-structure}\label{sec-application-of-the-structure-thm}

This section is devoted to the proof of our upper bounds for planar graphs and
graphs with bounded genus. The proof is based on results for stacked
triangulations, in particular Theorem~\ref{thm:pb+p-bounds}\ref{enu-stackedtri-up}.

We need some additional notation: By $G\boxtimes H$ we denote the graph with
vertex set $V(G)\times V(H)$ which has an edge between $(u,w)$ and $(u',w')$
if $uu'\in E(G)$ and $ww'\in E(H)$, or $u=u'$ and $ww'\in E(H)$, or $uu'\in E(G)$ and $w=w'$.

Let $G^p$ be the $p$th power of $G$, i.e., a pair $uw$ is an edge of $G^p$ whenever the distance
between $u$ and $w$ in $G$ is at most $p$. Note that any proper coloring of $G^p$  is a $p$-centered 
coloring of $G$, because any two vertices of the same color have distance at least $p+1$ in~$G$
and the $p+1$ first vertices on a path from one of these vertices to the other form a clique in
$G^p$, thus they receive pairwise different colors. Actually, this shows that these colorings admit an
even stronger condition than $p$-centered colorings: \emph{Any connected subgraph~$H$ in such a
coloring either contains more than $p$ colors or it contains every color at most once}.
We will use this property at the end of the proof of the following lemma.

\begin{lemma}
$$\chi_p(H_1\boxtimes H_2)\leq \chi_p(H_1)\cdot \chi(H_2^p) $$
\end{lemma}
\begin{proof}
  Let $\psi_1$ be a $p$-centered coloring of $H_1$ and $\psi_2$ be a proper
  coloring of $H_2^p$.  With~$\pi_1$ and $\pi_2$ we denote the the projections
  from $V(H_1)\times V(H_2)$ to its components.  We claim that the product
  coloring $\phi(v,w)=(\psi_1(v),\psi_2(w))$ is a $p$-centered coloring of
  $G=H_1 \boxtimes H_2$. Consider a connected subgraph $G'$ of $G$.  From the
  definition of $\boxtimes$ it follows that $H_1':=\pi_1(G')$ is
  connected. Since $\psi_1$ is a $p$-centered coloring of $H_1$, it either uses more
  than $p$ colors on $H_1'$ or there is a vertex $\hat{v}\in H_1'$, whose
  color is unique in $H_1'$. In the first case, $\phi$ uses more than $p$
  colors on $G'$ and we are done. In the second case consider a vertex
  $(\hat{v},w)\in G'$. The color of $(\hat{v},w)$ is unique in $G'$ unless
  there is a $w'\in H_2$ with $w\neq w'$, $\psi_2(w)=\psi_2(w')$, and $(\hat{v},w')\in G'$.
  If such a vertex $w'$ exists, then let $P$ be a path in~$G'$ from $(\hat{v},w)$ to
  $(\hat{v},w')$. Since not all colors in $\pi_2(P)$ are unique and $\pi_2(P)$ is connected in $H_2$, 
  $\psi_2$ uses more
  than $p$ colors on $\pi_2(P)$, whence $\phi$ uses more than $p$ colors on
  $P\subset G'$.
\end{proof}

Let $G+H$ denote the complete join of the graphs $G$ and $H$, this is the
graph whose vertex set is the disjoint union $V(G) \overset{\cdot}{\cup} V(H)$
the edge set contains $E(G)$ and $E(H)$ together all possible edges with one
end in $V(G)$ and one in $V(H)$.  The following theorem is due to
Dujmovi\'c et al.\ \cite{DJMMUW-queues}. 

\begin{theorem}[\cite{DJMMUW-queues}, Theorem 37]
\label{thm:genusGproduct}
Every graph of Euler genus $g$ is a subgraph of:
\begin{enumerate}
\item\label{gGp-first}
  $H\boxtimes P \boxtimes K_{\max(2g,3)}$ for some apex graph $H$ of
  treewidth at most $4$ and a path $P$.
\item\label{gGp-second}
  $(K_{2g}+H_8)\boxtimes P$ for some planar graph $H_8$ of treewidth at
  most $8$ and a path~$P$.
\end{enumerate}
\end{theorem}

The proof given in  \cite{DJMMUW-queues} actually reveals that
the graph $H$ in Theorem~\ref{thm:genusGproduct}(\ref{gGp-first})
is an apex graph over a stacked triangulation, i.e., removing the apex 
leaves a planar graph $H'$ with treewidth at most~3. Hence, with
Theorem~\ref{thm:pb+p-bounds}\ref{enu-stackedtri-up} we get
\begin{equation}
\chi_p(H) \leq \chi_p(H')+1 \in  \Oh(p^2 \log p). \tag{$*$}\label{Hprime}
\end{equation}

\begin{corollary}\label{cor:p+g}
For any planar graph $G$ and any graph $G_g$ of Euler genus $g$, we have:
\begin{align*}
\chi_p(G)&\in \Oh(p^3 \log p),\\
\chi_p(G_g)&\in \Oh(g p^3\log p),\\
\chi_p(G_g)&\in \Oh(gp + p^4\log p).
\end{align*}
\end{corollary}
\begin{proof}
  We begin with three easy observations.
  The $p$th power of the complete graph is just the complete graph, in
  particular $\chi(K_n^p)=\chi(K_n)=n$. A periodic coloring of a path $P$ with
  $p+1$ colors shows that $\chi(P^p)=p+1$. Furthermore, the additon of an apex
  to a graph can increase its $p$-centered chromatic number by at most~1. With
  Theorem~\ref{thm:genusGproduct} it now follows that
\begin{align*}
\chi_p(G)&\leq \chi_p(H\boxtimes P \boxtimes K_3)\leq \chi_p(H\boxtimes P) \cdot \chi(K_3^p)\leq \chi_p(H) \cdot \chi(P^p)\cdot 3 \in \Oh(p^3 \log p)\\
\chi_p(G_g)&\leq \chi_p(H\boxtimes P \boxtimes K_{2g+3})\leq \chi_p(H) \cdot \chi(P^p)\cdot (2g+3) \in \Oh(g p^3\log p)\\
\chi_p(G_g)&\leq \chi_p(K_{2g}+H_8) \cdot \chi(P^p)\leq (2g+\chi_p(H_8))\cdot (p+1)\in \Oh(gp + p^4\log p)
\end{align*}
For the final inequality of the first two lines we have used (\ref{Hprime}).
In the last line, we used that $H_8$ is planar and therefore by the first
line has a $p$-centered coloring using at most $\Oh(p^3\log p)$ colors.
\end{proof}

A coloring procedure which follows the proof of
Theorem~\ref{thm:genusGproduct} more closely can be used to show that graphs
of Euler genus $g$ actually admit $p$-centered colorings with
$\Oh(gp+p^3\log p)$ colors.  The proof for this slightly improved bound has
been detailed in the conference version of this paper~\cite{DFMS20}.  There it
is also indicated how the proofs can be turned into quadratic time algorithms
for colorings which respect the given bounds for the number of colors. The
bottleneck of the coloring algorithm is the computation of the product
structure.  Recently Morin~\cite{Morin21} has shown that a
$(K_{2g}+H_8)\boxtimes P$ structure can be constructed in~$\Oh(n\log n)$
time , this yields a corresponding improvement in the running time of the coloring
algorithm.

\section{Upper bound for bounded degree:
   The entropy compression method}
\label{sec-bounded-degree}
\def\uncol{\diamondsuit}
\def\LOG{\textsc{Log}}
\def\pos{\textrm{pos}}
\def\colv{\gamma}

In this section we prove Theorem~\ref{thm:bounded-degree}.

Let $p$ be a positive integer and let $G$ be a graph with maximum degree at
most $\Delta$.  Fix an arbitrary ordering $\Pi=(v_1,\ldots,v_n)$ of $V(G)$.
For every vertex $v\in V(G)$, fix an arbitrary ordering of the edges adjacent
to $v$.  Let $c=\left\lceil2^{10}\cdot\Delta^{2-1/p}\cdot p\right\rceil$ and
let $M$ be a sufficiently large integer divisible by $2p$.

A \emph{partial coloring} of $G$ is a function $f: V(G)\to C\cup\set{\uncol}$ 
where $C$ is a set of \textrm{colors} used by~$f$ and~$\uncol$ is an extra value 
indicating that no color is assigned.
A vertex $v$ with $f(v)\in C$ is said to be \emph{colored} by~$f$ 
while a vertex $v$ with $f(v)=\uncol$ is \emph{uncolored} by $f$.
A connected subgraph $H$ of $G$ is a \emph{violator} with respect to 
a partial coloring $f$ if:
\begin{enumerate}
\item all vertices in $H$ are colored by $f$;
\item $f$ uses at most $p$ colors on $H$; and
\item no color is unique in $H$ under $f$.
\end{enumerate}

Algorithm~\ref{algo-naive} below is a naïve randomized algorithm 
which tries to color $G$ in a $p$-centered fashion with $c$ colors.
The algorithm starts with all vertices uncolored.
In every step, the first uncolored vertex from $\Pi$ is colored with 
a random color from $\set{1,\ldots,c}$.
If this new assignment creates a violator $H$, then 
$\min(|V(H)|,2p)$ vertices in $H$ are uncolored, 
including the vertex colored in this step.
(One could erase the colors of all the vertices in $H$ and get a $\Oh(\Delta^2p)$ bound; 
we present this slightly more involved version to get $\Oh(\Delta^{2-1/p}p)$.)

To trade probabilistic arguments for counting we
replace random sampling by a prescribed sequence $X=(x_1,\ldots,x_M)$ of colors,
where $x_i\in [c]$ for every $i\in[M]$,
and take $X$ as part of the input of the algorithm,
see Algorithm~\ref{algo-naive}.

\begin{algorithm}
\caption{Input: $X=(x_1,\ldots,x_M)$ with $x_i\in\set{1,\ldots,c}$}
\label{algo-naive}
\begin{algorithmic}[1]
\State $f(v) \gets \uncol$, for every $v$ in $V(G)$
\For{$i = 1,2,\ldots, M$}\label{step_mainLoop}
  \If{all vertices of $G$ are colored under $f$}\label{step_breakingIf}
    \State\label{step_output} \textbf{exit}
  \EndIf
  \State\label{step_TakeUncoloredVertex} let $w_i$
                      be the first uncolored vertex in $\Pi$
  \State\label{step_assignColor} $f(w_i) \gets x_i$
  \If{there exists a violator of $f$}\label{step_mainIf}
  \begingroup
  \addtolength{\jot}{-.25em}
  \begin{align*}
  H_i &\gets \textrm{a violator of $f$}\\
  T_i &\gets \textrm{a spanning tree of $H_i$ rooted at $w_i$}\\
  W_i &\gets \textrm{a walk of $T_i$ that starts and ends in $w_i$, and traverses each}\\
  &\qquad\textrm{edge of $T_i$ twice}\\
  f(v) &\gets\uncol,\ \textrm{for every $v$ among the first $\min(|V(H_i)|,2p)$ distinct}\\
  &\qquad\textrm{vertices along $W_i$}
  \end{align*}
  \endgroup
  \EndIf
\EndFor
\State \textbf{return} $f$
\end{algorithmic}
\end{algorithm}

If the algorithm breaks the loop at line~\ref{step_output} and outputs
a coloring $f$, then $f$ is a $p$-centered coloring of~$G$ using at
most $c$ colors. Indeed, we keep an invariant that after each
iteration there is no violator with respect to $f$.  This is obviously
true before the first iteration.  Now within step $i$, the newly
colored vertex~$w_i$ could create violators, but each of them must
contain $w_i$. When the algorithm uncolors some vertices in a single
violator, it always uncolors $w_i$ itself.
So no violators remain.

We define a structure called $\LOG(X)$ which accumulates data during the
execution of the algorithm.  We will show that unless input $X$ leads to a
break we can reconstruct $X$ from $\LOG(X)$, i.e., the map $X \mapsto \LOG(X)$
is injective. On the other hand, we will show that $\LOG(X)$ takes strictly
less than $c^M$ values.  This is a contradiction as there is no injective map
from a set of size~$c^M$ into a smaller one.  Thus there has to be an input
which leads to a break (line~\ref{step_output}), i.e., to a valid coloring.

We continue to define $\LOG$ and to show how to reconstruct $X$ from $\LOG(X)$
when the algorithm with input $X$ does not break at line~\ref{step_output}.

The data in $\LOG(X)$ is a tuple $(Z,\Sigma,\Gamma,F)$,
the individual components of  $\LOG(X)$ are described in the following.

In $Z=(z_0,z_1,\ldots,z_M)$ we record the number
of colored vertices during the process.
Specifically,
$z_i$ is the number of colored vertices after 
finishing the $i$-th step of the loop.
Note that $z_0=0$, $z_1=1$ and for~$i>1$
either $z_i = z_{i-1}+1$ or $z_i = z_{i-1}+1 - \min(h_i,2p)$, where
$h_i = |V(H_i)|$ is the number of vertices of the 
violator $H_i$. 

In $\Sigma$ we collect information about the violator subgraphs $H$
which get partially uncolored at some iteration of the loop.  

Suppose that in the $i$-th step of the algorithm a subgraph $H_i$ 
is recognized as a violator.
Let $T_i$ and $W_i$ be the structures fixed by the algorithm after line~\ref{step_mainIf},
so $T_i$ is a spanning tree of $H_i$ rooted at $w_i$ and 
$W_i=(u_0,u_1,\ldots,u_{2h_i-2})$ is a walk in $G$ traversing~$T_i$ such that 
the walk starts and ends at $w_i$, i.e., $u_0=u_{2h_i-2}=w_i$, and
each edge of~$T_i$ is traversed twice.
Let $W_i'$ be the walk in $G$ going along $W_i$ until~$W_i$ has visited
$m_i=\min(|V(H_i)|,2p)$ distinct vertices and after that $W_i'$
takes the direct path in $T_i$ back to the root $w_i$ where it ends.
If $|V(H_i)| \leq 2p$, then $W_i'= W_i$, otherwise $W_i'$ is a walk
travering a subtree $T_i'$ of $T_i$ with $2p$ vertices. In either case
the traversed tree has $m_i-1$ edges, whence $W'_i$ has length $2m_i-2$.
Let $(u'_0,u'_1,\ldots,u'_{2m_i-2})$ be the ordered list of vertices
visited by $W_i'$.
For $j\in\set{0,\ldots,2m_i-3}$, $u'_ju'_{j+1}$ is a \emph{forward step} 
if $u'_j$ is the parent of $u'_{j+1}$ in $T_i$.
Otherwise, $u'_{j+1}$ is the parent of $u'_j$ in $T_i$ and $u'_ju'_{j+1}$ 
is a \emph{backward step}.

We encode $W'_i$ in two lists.  The binary list
$B_i=(b_0,b_1,b_2,\ldots,b_{2m_i-3})$ records which steps of $W'_i$ are
forward and which are backward, i.e., $b_j=1$ if $u'_ju'_{j+1}$ is a forward
step and $b_j=0$ if $u'_ju'_{j+1}$ is backward.  Clearly, $m_i-1$ steps are
forward and $m_i-1$ steps are backward.  The list
$L_i=(\ell_1,\ldots,\ell_{m_i-1})$ with $\ell_k\in\set{1,\ldots,\Delta}$
collects data about the forward steps of $W'_i$.  If $u'_ju'_{j+1}$ is the
$k$-th forward step in $W'_i$, then $\ell_k$ is the unique number such that
$u'_ju'_{j+1}$ is the $\ell_k$-th edge in the fixed ordering of edges adjacent
to~$u'_j$.

Initially $\Sigma=(B,L)$ consists of two empty lists, and whenever
it comes to uncoloring a violator $H_i$ in the course of the algorithm
then $B_i$ is appended to $B$ and $L_i$ is appended to $L$.

In $\Gamma$ we collect information about the colors of the vertices
of violator subgraphs at the moment they are uncolored within the iteration process.

Again, suppose that in the $i$-th step of the algorithm a
subgraph~$H_i$ with $h_i=|V(H_i)|$ is partially uncolored, and let
$m_i=\min(h_i,2p)$.  Let $W'_i$ be the traversal of a subgraph of
$H_i$ fixed in the definitions of $B_i$ and $L_i$.  Let
$V_i=(v_1',\ldots,v_{m_i}')$ be the vertices of $W'_i$ sorted by their
first-appearance in $W'_i$.  Let $B'_i=(b_1',\ldots,b_{m_i}')$ be a
binary sequence such that $b_i'=1$ if
$f(v_i')\not\in\set{f(v_1'),\ldots,f(v_{i-1}')}$ and $b_i'=0$
otherwise.  Let $c_i=\sum b_j'$ that is $c_i$ is the number of colors
used by~$f$ at~$W'_i$ before it is uncolored.  Since
$V(W'_i)\subseteq V(H_i)$ and $H_i$ is a violator we have
$c_i \leq p$.  Let $d_i = m_i - c_i$.  We claim that $d_i\geq c_i$.
Indeed, if $h_i > 2p$ then $d_i=m_i-c_i = 2p-c_i \geq p$.  Otherwise,
$m_i=|V(H_i)|$ and since $H_i$ is a violator every color in $H_i$ is
used at least twice so $d_i =h_i - c_i\geq c_i$ as well.

Let
$V_i^{1} = (v^1_1,\ldots,v^1_{c_i})$ be the restriction of $V_i$ to
those indices $j$ with $b'_j=1$.  Let
$V_i^{0} = (v^0_1,\ldots,v^0_{d_i})$ be the restriction of $V_i$ to
those indices $j$ with $b'_j=0$.  Finally, we put
$A_i=(f(v^1_1),\ldots,f(v^1_{c_i}))$ and $P_i=(p_1,\ldots,p_{d_i})$
where $p_j$ is chosen so $f(v^0_j) = f(v^1_{p_j})$.

\def\oAP{{}^*\kern-3ptAP}
Initially $(B',A,P)$ consists of three empty lists, and whenever
it comes to uncoloring a violator $H_i$ in the course of the algorithm
then $B'_i$ is appended to $B'$, $A_i$ is appended to $A$, and
$P_i$ is appended to $P$. Finally, we  merge $A$ and $P$ into 
a single list $\oAP$ by reversing $A$ and appending $P$ 
(only for technical reasons to simplify the counting). 
This will yield $\Gamma=(B',\oAP)$.
\medskip

The last component $F$ of $\LOG(X)$ consists of the partial coloring
$f$ of $G$ when the algorithm stops, i.e., after the $M$ iterations of
the loop. More explicitly, $F = (f(v_1),f(v_2),\ldots,f(v_n))$ where 
the coloring~$f$ is taken from the output.

For the counting it will be convenient to have lists of length independent 
of $X$. Let~$\colv$ be the number of colored vertices in $F$, 
i.e., the number of $v_i$ with $f(v_i)\neq\uncol$.
Note that $M-\sum m_i = \colv$. 
We append $\colv$ entries to the lists $B'$, $\oAP$ from $\Gamma$
and $2\colv+2v$ entries to $B$ from $\Sigma$, 
where $v$ is the total number of violators.
Now $|B| = 2M$, $|B'|=M$, and $|\oAP|=M$.
The length of $L$ is $\sum (m_i-1) \leq (M - \gamma)\left(1-\frac{1}{2p}\right) \leq \left(1-\frac{1}{2p}\right)M$ and
we append elements to $L$ to make the length of
$L$ equal to this value.

\begin{klaim}
The mapping $X \to \LOG(X)$ is injective.
\end{klaim}

\begin{proof}
  We will show how to reconstruct the original input $X$
  from the value of $\LOG(X)$.
  The reconstruction is done in two
  phases. In the forward phase we reconstruct:
  \begin{enumerate}[label=\textup{(\roman*)},nosep]
  \item $(w_1,w_2,\ldots,w_M)$ where $w_i$ is the vertex colored with $x_i$
  at line~\ref{step_assignColor} in the $i$-th iteration;
  \item $(U_0,U_1,\ldots,U_M)$ where $U_i$ is the set of uncolored vertices 
  after the $i$-th iteration.
  \item $W'_i$ for every $i$ such that the algorithm encounters a violator in the $i$-th step, 
  where $W'_i$ is the traversal of the subgraph of $H_i$ with all uncolored vertices, 
  fixed in the definition of $\LOG(X)$.
  \end{enumerate} 
  Note first that $U_0=V(G)$.
  For the two lists in
  $\Sigma \in \LOG(X)$ we initialize $\pos(B)=0$ and $\pos(L)=0$.
  The invariant is that when $U_{i-1}$ has been computed $\pos(B)$ and
  $\pos(L)$ mark the tail of the lists constructed up to the end of 
  the $(i-1)$-st iteration of the loop.

  Let $i\geq1$ and suppose that $U_{i-1}$ is determined.
  Then the vertex $w_i$ is identified as the 
  first element of~$U_{i-1}$ in ~$\Pi$.
  Now, if $z_{i} = z_{i-1} +1$ then 
  there was no violator in the $i$-th iteration step and we simply have 
  $U_i = U_{i-1} - \set{w_i}$.
  
  If $z_i \leq z_{i-1}$, then the coloring of $w_i$ led to a violator $H_i$
  whose vertices were partially uncolored. The number of uncolored vertices of $H_i$ 
  can be identified as $m_i=z_{i-1} +1 - z_i$. 
  Knowing $m_i$ we can separate the blocks 
  of the lists $B$ and $L$ corresponding to the subgraph $H_i$.
  Let $B_i=(b_0,b_1,b_2,\ldots,b_{2m_i-3})$ be the block of $2m_i-2$
  entries of $B$ starting from the current position $\pos(B)$ and advance
  $\pos(B) \gets \pos(B)+2m_i-2$. 
  Let $L_i=(\ell_1,\ldots,\ell_{m_i-1})$ be the block of $m_i-1$ entries
  of $L$ starting from the current position $\pos(L)$ and advance
  $\pos(L) \gets \pos(L)+m_i-1$. 
  Having identified $w_i$, $B_i$, and $L_i$ we can reconstruct 
  the traversal $W'_i=(u'_0,u'_1,u'_2,\ldots,u'_{2m_i-2})$ that was fixed 
  computing the value of $\LOG(X)$.
  Indeed, $u'_0=w_i$. We keep track of the number $k$ of the next forward step to come, so we start with $k=1$.
  When we have identified the vertex $u'_j$, for $j\geq0$, then the value
  $b_j$ indicates whether the next step is forward or backward. 
  If the step is forward, then we know that the $u'_ju'_{j+1}$ edge is 
  the $\ell_{k}$-th edge in the fixed ordering of edges adjacent to~$u'_j$.
  Therefore, we identified $u'_{j+1}$ and we increment $k$.
  If the step is backward, then we identify $u'_{j+1}$ 
  from the discovered so far subtree of $T_i$.
  When the sequence $W'_i=(u'_0,u'_1,u'_2,\ldots,u'_{2m_i-2})$ is reconstructed, 
  in particular, we have identified all the uncolored vertices of $H_i$.
  Finally, $U_i = U_{i-1} \cup \set{u'_0,u'_1,u'_2,\ldots,u'_{2m_i-2}}$.
  This completes the description of the forward phase.  
  
  In the backward phase we reconstruct:
  \begin{enumerate}[label=\textup{(\roman*)},nosep]
    \item $(f_M,f_{M-1},\ldots,f_1)$ where $f_i$ is the partial coloring $f$ 
    after the $i$-th iteration; and 
    \item $(x_M,x_{M-1},\ldots,x_1)$.
  \end{enumerate}
  Note that $f_M=F$ is a part of the evaluation of $\LOG(X)$.
  In particular, the number $\gamma$ of vertices colored after the whole process is known.
  For the two lists in $\Gamma \in \LOG(X)$
  we initialize $\pos(B') = M-\colv$, $\pos(A)=0$, and
  $\pos(P)=M  - \colv$.
  The invariant is that
  when $f_i$ has been computed these functions mark the
  tail of the lists after the $(i-1)$-st iteration of the
  loop. In order to mark the tail of list $A$, we mark the head of the respective part of $\oAP$, because $A$ is reversed there. The tail of $P$ and $\oAP$ coincide except for the $\colv$ added elements.

  Let $1<i\leq M$ and suppose that $f_i$ has been identified.
  Now, if $z_{i} = z_{i-1} +1$ then 
  there was no violator in the $i$-th iteration step and we simply have 
  $x_i = f_i(w_i)$ and $f_{i-1}$ is obtained from $f_i$ by uncoloring $w_i$.

  If $z_i \leq z_{i-1}$, then the coloring of $w_i$ led to a violator
  $H_i$.  From the forward phase we know the traversal
  $W'_i=(u'_0,u'_1,u'_2,\ldots,u'_{2m_i-2})$ with $u'_0=w_i$ of a
  spanning tree of the subgraph of $H_i$ uncolored in the $i$th
  step. Let $V_i=(v_1',u_2',\ldots,v_{m_i}')$ be the uncolored vertices of~$H_i$
  sorted by their first-appearance in $W'_i$.  From $B'$ take the
  block $B'_i=(b_1',\ldots,b_{m_i}')$ of length $m_i$ preceding the
  current position.  The sequence~$B'_i$ guides us how to identify the
  split of $V_i$ into $V_i^{1} = (v^1_1,\ldots,v^1_{c_i})$ and
  $V_i^{0} = (v^0_1,\ldots,v^0_{d_i})$.  They are simply the
  restrictions of $V_i$ to those indices $j$ with $b'_j=1$ and
  $b'_j=0$, respectively.  As we see above, from the lengths of the
  identified sequences we can reconstruct the values of $c_i$ and
  $d_i$, i.e., $c_i=|V_i^1|$ and~$d_i=|V_i^0|$.  Now we can identify
  the next block $A_i= (\alpha_1,\ldots,\alpha_{c_i})$ of length~$c_i$
  in $\oAP$ by reversing the next $c_i$ elements after $\pos(A)$ and
  the block $P_i=(p_1,\ldots,p_{d_i})$ of size $d_i$ preceding the
  position $\pos(P)$ in~$\oAP$.  The list~$A_i$ contains the colors of
  vertices in~$V^1_i$, namely, $f_{i-1}(v^1_j) = \alpha_j$.  For
  vertices in $V_i^{0}$ we use the list~$P_i$ and have
  $f_{i-1}(v^0_j) = \alpha_{p_{j}}$.  Hence, we finally get
  $x_i=\alpha_1$ and after resetting $f_{i-1}(w_i)=\uncol$ we have
  reconstructed the coloring $f_{i-1}$.  It only remains to update
  $\pos(B') \gets \pos(B')-m_i$, $\pos(A) \gets \pos(A)+c_i$, and
  $\pos(P) \gets \pos(P)-d_i$.  This completes the description of the
  backward phase and the proof of the claim.
\end{proof}

We now come to the final task in the proof, we estimate the
size of the image of the $\LOG$-function.
Recall that $\LOG(X) = (Z,\Sigma,\Gamma,F)$, we will independently
estimate the number of values each of the four components can take.

Recall that $Z=(z_0,z_1,\ldots,z_M)$.  Based on $Z$, we build an
auxiliary sequence $S_Z = (s_1,\ldots,s_{2M})$ where each $s_i$ is
either $+1$ or $-1$.  The sequence is obtained from left to right. Set
$s_1=z_1-z_0=1$.  Assume that $z_0,\ldots z_j$ have been processed. If
$z_{j+1}-z_j=1$, then append $+1$ to $S_Z$ and if
$z_{j+1}-z_j=-r \leq 0$ append a single $+1$ and a sequence of~$r+1$
entries of $-1$ to $S_Z$, finally add a sequence of $\colv$ entries of
$-1$. Note that replacing~$Z$ by $S_Z$ in $\LOG(X)$ would not change
the encoded information. Indeed knowing $\colv$ from $F\in\LOG(X)$ the
final sequence can be removed and $z_i$ is obtained as the sum of the
$i$-th $+1$ entry and the block of $-1$ entries following it in the
truncated~$S_Z$.

In $S_Z$ the number of $+1$ entries and the number of $-1$
entries are exactly $M$ and
every prefix of $S_Z$ has a non-negative sum.  These
sequences are known as Dyck-paths and counted
by the Catalan number~$C_M$. It is well-known that 
$C_M$ is in $o(4^{M})$.

For $\Sigma$ and $\Gamma$ we consider the data accumulated
with a violator $H_i$ and its $m_i$ uncolored vertices.
In $\Sigma=(B,L)$ we add a binary list $B_i$ of length $2m_i-2$
and a list $L_i$ of length~$m_i-1$ with entries in $\set{1,\ldots,\Delta}$.
In $\Gamma=(B',\oAP)$ we add a binary list $B'_i$ of length~$m_i$,
a list $A_i$ of length $c_i$ with entries in $\set{1,\ldots,c}$,
and a list $P_i$ of length $d_i$ with entries in
$\set{1,\ldots,p}$. From the properties of a violator we have
$c_i+d_i=m_i$ and $d_i\geq c_i$.

Since we know the lengths of these lists, we can upper bound the
possibilities for~$B$,~$L$, and~$B'$ respectively by $2^{2M}$,
$\Delta^{\left(1-\frac{1}{2p}\right)M}$, and $2^M$. Lists $A$ and $P$
together have $M$ entries and for each $i$ we have $d_i\geq
c_i$. Since $c > p$ we only overcount when assuming that $A$ and $P$
are of the same length, therefore, the number of possibilities for
$\oAP$ is bounded by $c^{M/2}p^{M/2}$.

For the final coloring $F$ there are $(c+1)^n$ options.

This way we obtain an upper bound on the size of the image of the $\LOG$-function 
$$o(4^{M}) \cdot 2^{2M} \Delta^{\left(1-\frac{1}{2p}\right)M} 2^M c^{M/2}p^{M/2}(c+1)^n.$$
Since $\LOG$ is injective we get
\begin{align*}
o(4^{M}) \cdot 2^{2M} \Delta^{\left(1-\frac{1}{2p}\right)M} 2^M c^{M/2}p^{M/2}\cdot(c+1)^n &\geq c^M,\\
o(32^{M})\cdot \Delta^{\left(1-\frac{1}{2p}\right)M} p^{M/2}\cdot (c+1)^n &\geq c^{M/2},\\
o\left(\left(2^{10}\cdot\Delta^{2-1/p}\cdot p\right)^M\right)\cdot (c+1)^{2n} &\geq c^M.
\end{align*}

Since we have chosen $c=\left\lceil2^{10}\Delta^{2-1/p}p\right\rceil$, 
we see that for large enough
$M$ the inequality above cannot hold.
Therefore, for some input $X$ the algorithm outputs a
$p$-centered coloring of $G$ using $c$ colors.
This completes the proof. 

\medskip

Let $c=\left\lceil2^{10}\Delta^{2-1/p}p\right\rceil$. 
If the naïve randomized coloring algorithm
is allowed to use $2c$ colors it succeeds in finding a valid coloring
in expected $\leq 2n\log(2c+1) + 4$ iterations.
Indeed, let $M=2n\log(2c+1)$ and consider the probability that
after $M+t$ iterations the algorithm did not break. This probability
is upper bounded by
$\frac{c^{(M+t)/2}(2c+1)^{n}}{(2c)^{(M+t)/2}}\leq\frac{1}{2^{t/2}}$,
whence the expected number of iterations is upper bounded by
$M + \sum \frac{1}{2^{t/2}} < M+4$. 

A single iteration step can be done in $\Oh(n)$ time.  Indeed, the most time
consuming thing to do is to check whether after coloring $w_i$
there is a violator. This can be done by iterating over all subsets
$C\subset\set{1,\ldots,c}$ of size at most $p$.  For each such $C$, we can
determine in $\Oh(n)$ time if the component of the subgraph spanned by
vertices with colors from $C$ that contains $w_i$ is a violator.  Since the
number of such subsets is a function of $c$ and $p$ only, i.e., of $\Delta$
and $p$, we conclude that a single iteration can be done in~$\Oh(n)$ time.

Therefore, the expected runtime of the randomized algorithm is $\Oh(n^2)$.

\section{Upper bound for graphs excluding a fixed topological minor: 
Lifting through the structure theorem}
\label{sec-lifting}

In this section we prove Theorem~\ref{thm:topological-minors-sugar}.
We start with an informal statement of the structure theorem by Grohe
and Marx~\cite{GM15}.  For every graph $H$, every graph excluding~$H$
as a topological minor has a tree-decomposition with bounded adhesion
such that every torso either has bounded degree with the exception
of a bounded number of vertices, or excludes a fixed graph as a
minor.  Furthermore, such a decomposition for a graph~$G$ can be
computed in time $f(H)\cdot|V(G)|^{\Oh(1)}$ for some computable
function~$f$.

We proceed with all the necessary definitions.

A graph $H$ is a \emph{minor} of a graph $G$ if $H$ can be obtained from $G$ by 
deleting vertices, deleting edges, and contracting edges. 
A graph $H$ is a \emph{topological minor} of a graph $G$ if $H$ can be obtained 
from $G$ by deleting vertices, deleting edges, and contracting edges with 
at least one endpoint of degree $2$.

A \emph{tree-decomposition} of a graph $G$ is a pair $(T,\set{B_t}_{t\in V(T)})$ 
where $T$ is a tree and the sets $B_t$ $(t\in V(T))$ are subsets of $V(G)$ 
called \emph{bags} satisfying
\begin{enumerate}[nosep]
\item for each edge $uv\in E(G)$ there is a bag containing both $u$ and $v$, and
\item for each vertex $v\in V(G)$ the set of vertices $t\in V(T)$ with 
$v\in B_t$  induces a non-empty subtree of $T$.
\end{enumerate}
The \emph{treewidth} of a graph $G$ is the minimum $k$ such that
$G$ has a tree-decomposition with all bags of size at most $k+1$.
An \emph{adhesion set} in the tree-decomposition $(T,\set{B_t}_{t\in V(T)})$ 
of $G$ is a set $B_t \cap B_{t'}$ where $tt'\in E(T)$.
The \emph{adhesion} of a tree-decomposition is the maximum size of an adhesion set.
The \emph{torso} of a bag $B_t$ is the graph obtained from~$G[B_t]$ by adding all 
the edges between vertices in every adhesion set $B_t \cap B_{t'}$
where $tt'\in E(T)$.

We are ready to present the precise statement of the structure theorem for graphs avoiding 
a fixed graph as a topological minor.
\begin{theorem}[\cite{GM15}]\label{thm-structure-theorem}
For every integer $k\geq1$, there exist $a(k)$, $c(k)$, 
$d(k)$, and $e(k)$ such that the following holds:
If $H$ is a graph on $k$ vertices and 
$G$ avoids $H$ as a topological minor,
then $G$ has a tree-decomposition $(T,\set{B_t}_{t\in V(T)})$ with 
adhesion at most $a(k)$ such that for every $t\in V(T)$:
\begin{enumerate}[label=\textup{(\roman*)},nosep]
\item the torso of $B_t$ has at most $c(k)$ vertices of degree larger than $d(k)$, or
\item the torso of $B_t$ avoids $K_{e(k)}$ as a minor.
\end{enumerate}
\end{theorem}

To handle tree-decompositions of bounded adhesion we use 
the following lemma from~\cite{PS19}.
We say that a tree-decomposition $(T,\set{B_t}_{t\in V(T)})$ of a graph $G$ is \emph{over 
a class of graphs} $\calC$ if for every $t\in V(T)$ 
the torso of $B_t$ lies in $\calC$.

\begin{lemma}[\cite{PS19}]\label{lemma-lifting-through-bounded-adhesion}
  Let $\calC$ be a class of graphs such that for a fixed $d$ and every
  $p\geq1$, every graph in $\calC$ admits a $p$-centered coloring with
  $\Oh(p^d)$ colors.  If $\calD$ is a class of graphs such that every
  graph $G\in\calD$ has a tree-decomposition over $\calC$ with
  adhesion at most $k$, then for every $p\geq1$, every graph in
  $\calD$ admits a $p$-centered coloring with $\Oh(p^{d+k})$ colors.
\end{lemma}

We also need the following theorem from \cite{PS19}.

\begin{theorem}[\cite{PS19}]\label{theorem-proper-minor-closed}
For every graph $H$ there is a polynomial $f$ such that 
the graphs excluding~$H$ as a minor admit $p$-centered
colorings with at most $f(p)$ colors.
\end{theorem}

With these tools in hand we can give a proof of a polynomial upper bound 
for graphs avoiding a fixed topological minor (Theorem~\ref{thm:topological-minors-sugar}).

Let $H$ be a graph on $k$ vertices and 
let $G$ be a graph avoiding $H$ as a topological minor. From
Theorem~\ref{thm-structure-theorem} we obtain a tree-decomposition 
$(T,\set{B_t}_{t\in V(T)})$ with adhesion at most $a(k)$ such that
every torso of a bag in the tree-decomposition satisfies one of the two 
properties mentioned there.

For fixed $p\geq1$, we construct a $p$-centered coloring of $G$. 
Let $t\in V(T)$ and consider the torso $\tau(t)$ of the bag $B_t$.
Suppose that $\tau(t)$ has at most $c(k)$ vertices of degree larger than $d(k)$.
We color the vertices of $\tau(t)$ as follows:
\begin{enumerate}
\item use at most $c(k)$ distinct colors on vertices of degree larger than $d(k)$;
\item color the remaining vertices using at most $(32\cdot d(k))^2\cdot p$ colors 
applying Theorem~\ref{thm:bounded-degree}.
\end{enumerate}
Clearly, we obtain a $p$-centered coloring of~$\tau(t)$ using at most
$c(k)+(32\cdot d(k))^2\cdot p$ colors.

Now suppose that $\tau(t)$ avoids $K_{e(k)}$ as a minor.  By Theorem
\ref{theorem-proper-minor-closed}, there is a constant~$f(k)$ such that
$\tau(t)$ admits a $p$-centered coloring using $\Oh(p^{f(k)})$ colors.

Therefore, by Lemma~\ref{lemma-lifting-through-bounded-adhesion}, the graph $G$ 
admits a $p$-centered coloring using $\Oh(p^{1+f(k)+a(k)})$ colors.
This completes the proof.

As mentioned, Theorem~\ref{thm-structure-theorem} comes with an algorithm 
that produces a desired tree-decomposition for a graph $G$ on $n$ vertices in
$f(H)\cdot n^{\Oh(1)}$ in time for some computable function $f$. Indeed, both results
from~\cite{PS19},  Lemma~\ref{lemma-lifting-through-bounded-adhesion} and
Theorem~\ref{theorem-proper-minor-closed}, are also supported by algorithms
with  running times $n^{\Oh(1)}$. The final algorithm within the proof of our theorem
calls the algorithm for bounded degree case given by 
Theorem~\ref{thm:bounded-degree}.
As we have seen in the previous section, the expected running
time of the latter algorithm is~$\Oh(n^2)$. All in all, there is a randomized polynomial time
algorithm to determine a $p$-centered coloring of a graph without  $H$ as a topological
minor with polynomial in $p$ number of colors.

\section{Upper bounds for graphs of bounded simple treewidth}
\label{sec-upper-bounds-subplanars}

Theorem \ref{thm:planarup} has been shown as a part of
Corollary~\ref{cor:p+g}.  There we made use of an upper bound for planar
graphs of treewidth at most 3
(Theorem~\ref{thm:planarclasses}.\ref{enu-stackedtri-up}).  We remark that
planar graphs of treewidth at most 3 are also known as subgraphs of stacked
triangulations, see~\cite{KV12}. We proceed to prove
Theorem~\ref{thm:planarclasses}.\ref{enu-outerplanar-up} and
Theorem~\ref{thm:simpletree}.\ref{enu-simpletree-up} which also gives
Theorem~\ref{thm:planarclasses}.\ref{enu-stackedtri-up} To establish these
results, we need a structural property of chordal graphs, sometimes called the
\emph{shadow completeness}, that was independently observed by Dujmovi\'c,
Morin and Wood~\cite{DMM05} and by Kündgen and Pelsmajer~\cite{KP08}.
  
\begin{lemma}[shadow completeness]\label{lemma-shadow}
  Let $G$ be a chordal graph and let $(V_0,V_1,\ldots)$ be a BFS-layering
  of~$G$. Let $i>0$ and let $C\subseteq G[\bigcup_{j> i} V_j]$ be
  connected. Let the \emph{shadow} $S_i(C)$ of $C$ in $V_i$ be the set of
  $v\in V_i$ such that there is a path $P$ from $v$ to $C$ with
  $P\cap V_i = \{v\}$. Then $S_i(C)$ is a clique. Moreover, $S_i(C)$ separates
  $C$ from $G[\bigcup_{j< i} V_j]$.
\end{lemma}

\begin{proof}
  Let $C$ be a connected subgraph of $G[\bigcup_{j> i} V_j]$.  Let $v$
  and $v'$ be two vertices in~$S_i(C)$.  Let $P$ and $P'$ be the
  certifying paths of $v$ and $v'$. The ends of $P$ and $P'$ can be
  connected in $C$ yielding a path connecting~$v$ and $v'$ which has
  all inner vertices in $G[\bigcup_{j> i} V_j]$. Let $P_1$ be a
  shortest path from $v$ to $v'$ with all inner vertices in
  $G[\bigcup_{j> i} V_j]$. Let $P_2$ be a shortest path from $v$ to
  $v'$ with all inner vertices in $\bigcup_{j< i} V_j$ (there is one
  as $G\left[\bigcup_{j< i} V_j\right]$ is connected in the BFS-tree).
  Now $P_1 \cup P_2$ is a cycle in $G$ of length at least~$4$ and the
  only possible chord is an edge $vv'$.  Since $G$ is chordal, we have
  $vv'\in E(G)$, as desired.

  Since every path from a vertex in $\bigcup_{j< i} V_j$ to a vertex
  in $C$ contains a vertex in $V_i$ the shadow $S_i(C)$ is a separator
  for these two sets.
\end{proof}

\begin{corollary}
  Let $G$ be a chordal graph and let $(V_0,V_1,\ldots)$ be a BFS-layering of
  $G$. Let $i\geq0$ and let $H\subseteq G[\bigcup_{j\geq i} V_j]$ be a
  connected induced subgraph. Then $H\cap G[V_i]$ is connected.
	\label{Cor:shadow}
\end{corollary}
\begin{proof}
  Consider two vertices $u,v$ in $H\cap G[V_i]$.  Take a path $P$ from $u$ to
  $v$ in $H$.  If all vertices of $P$ are in $H\cap G[V_i]$, then we are done.
  Otherwise, consider a vertex $x$ of $P$ with $x\in V_i$ such that the vertex
  after~$x$ in~$P$ is not in $V_i$.  Let $y$ be the first vertex of~$P$ after
  $x$ with $y\in V_i$. Lemma~\ref{lemma-shadow} implies that $x$ and $y$ are
  adjacent. Thus $P$ is not a shortest path between $u$ and $v$ in $H$, a
  contradiction.
\end{proof}

\subsection{Proof of the upper bound for outerplanar graphs}\hfill
	
  Let $G$ be an outerplanar graph. Then $G$ is a subgraph of a
  maximal outerplanar graph $G^+$. Since any $p$-centered coloring of
  $G^+$ can be restricted to $G$, we may assume $G=G^+$. It is
  well-known that maximal outerplanar graphs are 2-trees. 
  In particular,~$G$ is chordal. Let $(V_0,V_1,\ldots)$ be a BFS-layering
  of $G$.
\begin{klaim}
  The subgraph $G[V_i]$ of $G$ induced by $V_i$ is a linear forest,
  i.e., each connected component of it is a path.
\end{klaim}
\begin{proof}
  Contracting $G[\bigcup_{j< i} V_j]$ and deleting
  $G[\bigcup_{j> i} V_j]$ yields a minor $M$ of $G$ that consists of
  $G[V_i]$ and an additional apex $a$ adjacent to every $v\in V_i$. 
  Since adding an apex increases
  the treewidth by exactly one, we have $\tw(G[V_i])\leq 1$, so $G[V_i]$ is a
  forest. Assume $v\in G[V_i]$ has (at least) 3 neighbors
  $x,y,z\in G[V_i]$. Deleting the edge $av$ and all vertices except
  $x,y,z,v,a$ from $M$ yields a $K_{2,3}$-minor of $G$, 
  which is a contradiction with $G$ being outerplanar.
  Thus $G[V_i]$ is a forest with all vertices of degree at most $2$, 
  as desired.
\end{proof}

We will refer to the connected components of $G[V_i]$ as the
\emph{layer paths} of $G[V_i]$.
In the following, we construct a coloring
$\phi:V(G) \rightarrow C$ with $C$ being a set of colors 
of size $p\lceil\log(p+1)\rceil + 2p+1$.

We construct $\phi$ layer by layer. 
In each layer we color the layer paths independently. When it comes
to coloring a layer path $P$ we have a set $\calF(P)$ of colors which
are forbidden for $P$.
We initialize $\calF(P)=\emptyset$ for every layer path $P$ in $G$.
The set $\calF(P)$ will be of size at most $p\lceil\log (p+1)\rceil+p$.
The path $P$ will be colored with a set of $p+1$ colors from $C-\calF(P)$.

The layer $V_0$ contains only a single vertex $v$. 
Set $\phi(v)=\alpha$ for some color $\alpha\in C$. 
To make
sure that $\alpha$ does not appear on the next $p$ layers we add $\alpha$ 
into $\calF(P)$ for every layer path $P$ of $G[V_i]$ with
$0 < i \leq p$.

If all layers $V_j$ with $0\leq j <i$ are colored and $P$ is a layer path 
of $G[V_i]$, then we extend~$\phi$ to $P$ as follows: 
Choose a set $C_P \subset C-\calF(P)$ of $p+1$ colors, say 
$C_P =\set{0_P,1_P,\ldots,p_P}$.  
Color $P=(v_0,v_1,\ldots,v_m)$ periodically, i.e., let
$\phi(v_k)=j_P$ if $k=j \modu\ (p+1)$.  

Now consider a layer path $P'$ in $G[V_j]$ with $i<j$, and note that
the
shadow $S_i(P')$ of $P'$ in $V_i$ is the set
\[
\left\{v\in V_i\mid \hbox{$v$ is adjacent to the component of $G\left[\bigcup_{j'>i} V_{j'}\right]$ containing $P'$}\right\}.
\]
By Lemma~\ref{lemma-shadow} the set $S_i(P')$
induces a clique in $G[V_i]$.
Therefore, $S_i(P')$ is contained in one of the layer paths of
$G[V_i]$ and $|S_i(P')|\leq 2$. Moreover $S_i(P')$ is a separator.
For $i<k<j$ let $P_k$ be the layer path of $V_k$ containing
$S_k(P')$. The separator property implies the following
transitivity of shadows:
$S_i(P') = S_i(P_k)$. Indeed, a path from $V_i$ to $P'$
has to pass through $S_k(P')\subset P_k$ and every path from
$V_i$ to $P_k$ can be extended within $G[\bigcup_{j'\geq k} V_{j'}]$
to a path from $V_i$ to $P'$.

For each layer path $P'$ in $G[V_k]$ with $i<k\leq i+p$ such that
$S_i(P') \subseteq V(P)$ we extend $\calF(P')$ by some colors used on
$P$.  To determine these forbidden colors we use an auxiliary
structure.  Let $T$ be a binary tree on $p$ nodes with root $r_T$ and
height $\lceil \log(p+1)\rceil$ where height of a rooted tree is the
maximum number of nodes on a path from the root to a leaf.  We label
the nodes of $T$ with the totally ordered set of colors $1_P < 2_P\ldots <p_P$
using an \emph{in-order} traversal, i.e., the label of each node is
larger than all labels of nodes in its left subtree and smaller than
all labels of nodes in its right subtree. For
$\ell\in\set{1_P,\ldots,p_P}$ let~$F(\ell)$ denote the set of labels
of ancestors of the node~$u$ labeled $\ell$, i.e., $F(\ell)$ is the
set of labels seen on the path from $u$ to $r_T$ in $T$. We extend
this definition by $F(0_P)=\emptyset$.

Let $i<j\leq i+p$ and consider a layer path $P'$ of $G[V_j]$ with 
$S_i(P') \subset V(P)$.
We extend the set of forbidden colors of $P'$ putting
\[
  \calF(P') \gets \calF(P') \cup \set{0_P} \cup \hbox{$\bigcup_{v\in S_i(P')} F(\phi(v))$}.
\]
 \begin{klaim}\label{claim_inorder}
If $\ell,\ell'$ are consecutive in the cyclic order on $\set{0_P,1_P,\ldots,p_P}$, then
\[F(\ell)\subseteq F(\ell')\quad\textrm{or}\quad F(\ell')\subseteq F(\ell).
\]
 \end{klaim}
\begin{proof}
  If either $\ell$ or $\ell'$ is $0$, the statement is obvious.
  Otherwise $\ell'=\ell+1$ and there are nodes $u$ and $u'$ with these
  labels in $T$. 
  It is easy to see that every two consecutive nodes in an in-order traversal 
  must be in the ancestor-descendant relation in the tree.
  This immediately gives $F(\ell)\subseteq F(\ell')$ or $F(\ell')\subseteq F(\ell)$.
\end{proof}
Note that if $S_i(P')=\set{v,v'}$, then $v$ and $v'$ are consecutive on $P$.
By the claim proven above we have
\[
\left|\hbox{$\bigcup_{v\in S_i(P')} F(\phi(v))$}\right| \leq \lceil\log(p+1)\rceil.
\]
Hence at most $\lceil\log(p+1)\rceil + 1$ colors are added to $\calF(P')$
after coloring the layer paths in $G[V_i]$.
Since~$\calF(P')$ is extended at most $p$ times we get an upper bound
\[
|\calF(P')| \leq p\lceil\log(p+1)\rceil + p.
\]
\begin{klaim}
  $\phi$ is a $p$-centered coloring of $G$.
\end{klaim}
\begin{proof}
  Let $H$ be a connected subgraph of $G$. We want to show that $\phi$
  uses more than $p$ colors on $H$ or there is a color used only once
  on $H$. Let $i:=\min\set{j\geq0\mid V_j\cap V(H)\neq \emptyset}$.  We distinguish two
  cases: \medskip

\noindent\textbf{Case 1.}
$H$ contains a vertex from $V_{i+p}$.

Let $Q$ be a layer path in $V_{i+p}$ containing a vertex $u_{i+p}$
from $H$. For all $j$ with $i\leq j < i+p$ the connectivity of~$H$
together with the separator property of $S_j(Q)$ imply that
$S_j(Q) \cap V(H) \neq \emptyset$. Therefore we can choose
$u_j \in S_j(Q) \cap V(H)$ for $i\leq j < i+p$.

We claim that $\phi(u_j) \neq \phi(u_k)$ for all
$i\leq j < k \leq i+p$.  Let $Q_k$ be the layer path containing
$u_k$. From $S_k(Q) \subset Q_k$ it follows from the transitivity
property of shadows that $S_j(Q_k) = S_j(Q)$, whence $u_j \in S_j(Q_k)$. Now
$\phi(u_j) \in F(\phi(u_j)) \subset \calF(Q_k)$,
Therefore, the color $\phi(u_j)$ is forbidden for the path $Q_k$ 
and $\phi(u_k)\neq\phi(u_j)$.
  
We conclude that the vertices $u_i,\ldots,u_{i+p}$ are colored with $p+1$ 
distinct colors. Therefore $\phi$ uses more than $p$ colors on~$H$.
\medskip
  
  \noindent\textbf{Case 2.}
  $H$ contains no vertex from $V_{i+p}$.
		
  Corollary \ref{Cor:shadow} implies that $Q=G[V(H)\cap V_i]$ is connected, so
  it is a subpath of some layer path $P$ of $G[V_i]$. Let $C_P=\set{0_P,...,p_P}$ be the set of colors
  used to color $P$.
  Let $C_Q\subset C_P$ be the set of colors used on $Q$.
  If two vertices of $Q$ have the same color, then due to the periodicity
  of $\phi$ on $P$ all the $p+1$ colors from $\set{0_P,\ldots,p_P}$ are used on $Q$, whence
  $\phi$ uses on $H$
  more than $p$ colors. From now on we assume that all colors on $Q$
  are unique.
  
  If $0_P \in C_Q$, then recall that $0_P \in \calF(P')$ for any layer path
  $P'$ of $G[V_k]$ with $i<k\leq i+p$ such that $S_i(P') \subseteq V(P)$.
  Therefore $0_P$ is contained in $\calF(P')$ for all $P'$ intersected
  by $V(H)$. Hence $0_P$ is a unique color on $H$. From now on we
  assume that $0_P \not\in C_Q$.
		
  Consider the rooted binary tree $T$ on $p$ nodes devised during the coloring procedure 
  of $P$. Since $C_Q$ is an interval of $\set{1,\ldots,p}$ 
  there is a unique node $r_Q$ in $T$ with a label from $C_Q$ such
  that $r_Q$ is a predecessor of all nodes with labels from $C_Q$ in $T$.
  Let $\beta$ be the label
  of $r_Q$. It follows that $\beta \in F(\ell)$ for all $\ell\in C_Q$, whence
  $\beta\in F(\phi(v))$ for all $v\in Q$.  Therefore $\beta \in \calF(P')$
  for any layer path
  $P'$ of $G[V_k]$ with $i<k\leq i+p$ such that $S_i(P') \subseteq V(P)$.
  This means that $\beta$ is unique on $H$, as desired.
\end{proof}
In summary, any outerplanar graph has a $p$-centered coloring with
$p\lceil \log(p+1)\rceil + 2p +1 \in \Oh(p\log p)$. This completes the
proof of Theorem \ref{thm:planarclasses}.\ref{enu-outerplanar-up}.
The described algorithm coloring outerplanar graphs has a straightforward $\Oh(n)$
time implementation.

\subsection{Proof of the upper bound for graphs of bounded simple treewidth}
\label{subsec:simpletree-upper}

For $p\geq1$, let $f(p)$ be the maximum $\chi_p(G)$ when $G$ is outerplanar.
We have just seen that $f(p)\in \Oh(p\log p)$.  Now we are going to show that
every planar graph of treewidth at most $3$ admits a $p$-centered coloring
with at most $(p+1)f(p)$ colors. This proof can be generalized though, which
is what we will do here. The essence of the generalization is that maximal
outerplanar graphs are simple $2$-trees.  We continue with a discussion of
simple treewidth.

For $k\geq1$, a tree-decomposition $(T,\calB)$ of $G$ is $k$-\emph{simple} 
if $(T,\calB)$ is of width at most~$k$ and for every $X\subset V(G)$ with $|X|=k$, 
we have that $X\subset B_t$ for at most two distinct~$t\in V(T)$. 
The \emph{simple treewidth} of~$G$, denoted by $\stw(G)$, 
is the least integer $k$ such that $G$ has a $k$-simple tree-decomposition. 

The notion of simple treewidth was proposed by Knauer and Ueckerdt~\cite{KU},
where graphs of $\stw(G)$ are defined as subgraphs of simple $k$-trees.  Since
we do not want to discuss $k$-trees and their construction order we stick to
our equivalent definition, see~\cite{W16} for further discussion.  It is easy
to see that $\tw(G)\leq \stw(G)\leq \tw(G)+1$, for every graph $G$.  Graphs
of simple treewidth at most~$1$ are disjoint unions of paths.  Graphs of
simple treewidth at most~$2$ are outerplanar graphs.  Graphs of simple
treewidth at most~$3$ are planar graphs of treewidth at most $3$, i.e.,
subgraphs of stacked triangulations, see~\cite{KV12}.

The first thing to learn about simple treewidth is that it is monotone under taking minors, 
i.e., if $H$ is a minor of $G$, then $\stw(H) \leq \stw(G)$, see~\cite{JM-weak-cols}.
The next key fact is that the simple treewidth goes down when we restrict the graph to
a single BFS layer.

\begin{lemma}[\cite{JM-weak-cols}]
\label{lem:stw_layering}
Let $G$ be a connected graph with $\stw(G)=k \geq 1$ and let
$(V_0,V_1,\ldots)$ be a BFS-layering of $G$.  Then $\stw(G[V_i]) \leq k-1$ for
every $i\geq0$.
\end{lemma}

For $k\geq2$ and $p\geq1$, 
let $f_k(p)$ be the maximum $p$-centered chromatic number of graphs of simple treewidth at most $k$.
We will show that $f_{k+1}(p)\leq (p+1) f_k(p)$.
Given that $f_2(p)\in\Oh(p \log p)$ this yields $f_k(p)\in\Oh(p^{k-1}\log p)$ for all $k\geq2$.
Below we proceed with the induction step for $k\geq3$.

Let $G$ be a graph and $\stw(G)\leq k$.  Fix a tree-decomposition $(T,\calB)$
of $G$ witnessing the simple treewidth and let $G^+$ be a supergraph of $G$ so
that each $B\in\calB$ induces a clique in $G^+$.  Clearly, $\stw(G^+)\leq k$
and $G^+$ is chordal.  Since every $p$-centered coloring of $G^+$ is a
$p$-centered coloring of $G$ we work with~$G^+$.
Let $(V_0,V_1,\ldots)$ be a BFS layering of $G^+$.

Let $v$ be a vertex in $G^+$ with $v\in V_i$. We define $\alpha(v)=i\!\!\mod (p+1)$.
By Lemma~\ref{lem:stw_layering} $\stw(G^+[V_i])\leq k-1$.
Let $\beta_i$ be a $p$-centered coloring of $G^+[V_i]$ 
using at most $f_{k-1}(p)$ colors.
We define $\beta(v)=\beta_i(v)$.
Finally, for a vertex $v\in V(G^+)$ we define $\phi(v)=(\alpha(v),\beta(v))$. 
Clearly, $\phi$ uses at most $(p+1)\cdot f_{k-1}(p)$ colors.
We claim that $\phi$ is a $p$-centered coloring of $G^+$.

Let $H$ be a connected subgraph of $G^+$.
We want to show that either $\phi$ uses more than $p$ colors on~$H$ or
there is a color that appears exactly once on $H$.
Let $\ell$ be minimal such that $V(H)\cap V_{\ell}\neq\emptyset$.
The set $X=V(H)\cap V_{\ell}$ induces a connected subgraph of $G^+$ by
Corollary \ref{Cor:shadow} (here we are using that $G^+$ is chordal).
Since $\beta_{\ell}$ is a $p$-centered coloring of $G^+[V_{\ell}]$ we have that
either $|\beta(X)|>p$
or there is a vertex in $X$ 
of unique color under $\beta$.
In the first case, $\phi$ takes more than $p$ values on $V(H)$.
In the second case, fix the vertex $x\in X$ of the unique color under $\beta$.
If $x$ has a unique color under $\phi$ in $H$, then we are done.
Otherwise, let $x'\in V(H)$ be a vertex with $x\neq x'$ and $\phi(x)=\phi(x')$.
Let $x'\in V_{\ell'}$.
Since~$x$ has a unique color in $V(H)\cap V_{\ell}$, we get that $\ell\neq\ell'$.
Since $\alpha(x)=\alpha(x')$ we conclude that $\ell'-\ell\geq p+1$.
Since $x$ and~$x'$ are two vertices in a connected graph~$H$ 
we know that $H$ must intersect every layer $V_k$ with $\ell\leq k \leq \ell'$.
This means that~$\alpha$ takes all possible $p+1$ values on the vertices of $H$.
Therefore, $\phi$ uses at least $p+1$ colors on~$H$, as desired.

The proof implies an algorithm for coloring simple $k$-trees. This algorithm
may have many calls for the corresponding algorithm for $(k-1)$-trees but, 
since the graphs of these calls are disjoint, 
the overall time complexity remains $\Oh(n)$ for fixed $k$.

\section{Lower bound for graphs of bounded treewidth} 
\label{sec-lower-bounds-bounded-tw-linear}
\def\RT{R} 
\def\TT{T} 


Let $p$ be a positive integer. A vertex coloring $\phi$ of a graph $G$ is
\emph{$p$-linear} if for every path $P$ in $G$ either $\phi$ uses more than
$\phi$ colors on $P$ or there is a color that appears exactly once on $P$.
The \emph{$p$-linear chromatic number} $\lin_p(G)$ of $G$ is the minimum
integer $k$ such that there is a $p$-linear coloring of $G$ using $k$ colors.

Clearly, every $p$-centered coloring is a $p$-linear coloring. 
Thus, $\chi_p(G) \leq \lin_p(G)$ for every graph $G$.
We show below that 
for every $p\geq0$ and $t\geq0$, there is a graph $G$ of treewidth at most $t$ with 
\[
\lin_p(G) \geq \binom{p+t}{t}.
\]

We start with the key definition for our inductive construction.
For an integer $x$ with $x\geq1$, a rooted tree $\RT$ is \emph{$x$-ary}, 
if every non-leaf vertex of $\RT$ has degree at least $x$. 
The \emph{depth} of a vertex $v$ in a rooted tree~$\RT$ 
is the number of vertices on the path between $v$ and the root in $\RT$.
A rooted tree $\RT$ has \emph{depth} $d$ if every leaf in~$\RT$ has depth $d$.
A subgraph $\RT$ of a graph $G$ is $\emph{$(p,d,x)$-brushed}$ if:
\begin{enumerate}[label=\textup{(\roman*)},ref=\textup{\roman*},nosep]
  \item $\RT$ is an $x$-ary tree of depth $d$; and
  \item for every vertex $v$ in $\RT$,
  if $i\in[d]$ is the depth of $v$ in $\RT$, then
  there exist indices $j_1,\ldots,j_k$ with $1\leq k \leq p$ and $i=j_1 < \cdots < j_k=d$ such that for every leaf $w$ of $\RT$ that lies in the subtree of $v$, 
  if $u_i\cdots u_d$ is a path in $\RT$ connecting $u=u_i$ with $w=u_d$, then
  the sequence $u_{j_1}\cdots u_{j_k}$ induce a path in~$G$.
  \label{def-brushed-ii}
\end{enumerate}

We construct a family of graphs $\set{G_{(p,t,x,N)}\mid p\geq0,\ t\geq0,\ x\geq2,\ N\geq1}$ such that the following invariants hold:
\begin{enumerate}
\item $\tw(G_{(p,t,x,N)}) \leq t$; and
\label{enu-invariant-tw-linear-version}
\item for every integer $p'$ with $p'\geq p$ and every $p'$-linear coloring $\phi$ of $G_{(p,t,x,N)}$ either:
\label{enu-invariant-general-for-colorings}
\begin{enumerate}
\item $\phi$ uses at least $N$ colors on $G_{(p,t,x,N)}$, or 
\label{enu-invariant-N-colors}
\item 
$G$ has a $\left(p,\binom{p+t}{t},x\right)$-brushed subgraph $\RT$ and
there is a sequence $\left(\alpha_1,\ldots,\alpha_{\binom{p+t}{t}}\right)$ of distinct colors so that 
for every $i\in\left\{1,\ldots,\binom{p+t}{t}\right\}$ and every vertex $v$ of depth $i$ in $\RT$, we have $\phi(v)=\alpha_i$.
\label{enu-invariant-x-ary-tree-linear-version}
\end{enumerate}
\end{enumerate}
This family, in particular $G_{\left(p,t,2,\binom{p+t}{t}\right)}$, will witness the statement of the theorem.

We start the construction with the base cases. 
Let $x\geq2$ and $N\geq1$.
For every $p\geq0$ and every $t\geq0$ the graphs $G_{(0,t,x,N)}$ and $G_{(p,0,x,N)}$ contain just a single vertex.
Since these graphs have no edges, their treewidth is $0\leq t$, so invariant \eqref{enu-invariant-tw-linear-version} holds.
A single vertex is also an $x$-ary tree of depth $1$, so invariant \eqref{enu-invariant-x-ary-tree-linear-version} holds.

For the inductive step let $p\geq1$, $t\geq1$, $x\geq2$, and $N\geq1$. 
Let $M = \binom{p+t-1}{t-1}$ and $X = (x-1)N^M +1$.
The graph $G_{(p,t,x,N)}$ is obtained from a copy $G^0$ of $G_{(p-1,t,X,N)}$
by adding for every vertex $v$ of $G^0$ disjoint copies $G^{v,1},\ldots,G^{v,X}$ of
$G_{(p,t-1,x,N)}$ in such a way that each vertex of $G^{v,i}$ with $i\in[X]$ is adjacent to
$v$ and to no other vertices outside~$G^{v,i}$.

First, we argue that $\tw(G_{(p,t,x,N)})\leq t$. Take a tree-decomposition
$(\TT^0,\calB^0)$ of $G^0$ of width $\leq t$ and for each $v\in V(G^0)$ and each
$i\in[X]$ take a tree-decomposition $(\TT^{v,i},\calB^{v,i})$ of $G^{v,i}$ of
width $\leq t-1$.  We construct a tree-decomposition $(\TT,\calB)$ of
$G_{(p,t,x,N)}$ as follows.  Let $V(\TT)= V(\TT^0)\cup\bigcup V(\TT^{v,i})$.  The
edges of $\TT$ include all the edges of $\TT^0$ and all the edges of $\TT^{v,i}$ for
every $v\in V(G^0)$ and $i\in[X]$.  Additionally, for every $v\in V(G^0)$, fix
some $t_v\in V(\TT)$ such that $v\in B^0_{t_v}\in\calB^0$ and for each $i$ fix
some vertex $t^{v,i}$ of $\TT^{v,i}$.  In~$\TT$ the tree $\TT$ is connected to
$\TT^{v,i}$ with the extra edge $t_v t^{v,i}$. In simple words, we make a tree
out of the forest formed by previous trees.  The set $\calB$ of bags for $\TT$
are defined as follows: for every $t\in V(\TT^0)$ let $B_t = B^0_t$ where
$B^0_t \in \calB^0$, and for every $t \in V(\TT^{v,i})$ let
$B_t = B^{v,i}_t\cup\set{v}$.  It is elementary to verify that $(\TT,\calB)$ is
a tree-decomposition of $G_{(p,t,x,N)}$ and the width of $(\TT,\calB)$ is at
most $t$.

Let $p'\geq p$ and let $\phi$ be a $p'$-linear coloring of $G_{(p,t,x,N)}$.
We want to show that \eqref{enu-invariant-general-for-colorings} is satisfied for the coloring
$\phi$.  
First of all, if $\phi$ uses at least $N$ colors, then~ \eqref{enu-invariant-N-colors} holds.
Therefore from now on we assume that $\phi$ uses less than $N$ colors.

Since $G^0$ is isomorphic to $G_{(p-1,t,X,N)}$ we have that $G^0$ has a
$\left(p,\binom{p-1+t}{t},X\right)$-brushed subgraph that satisifes
\eqref{enu-invariant-x-ary-tree-linear-version}.  We fix a witnessing $X$-ary
tree $\RT^0$ of depth $\binom{p-1+t}{t}$ and the respective sequence of colors
$C^0=\left(\alpha_1,\ldots,\alpha_{\binom{p-1+t}{t}}\right)$ such that
$\phi(u)=\alpha_i$, for every 
vertex $u$ of depth $i$ in $\RT^0$.

Let $u$ be a leaf of $\RT^0$ and $j\in [X]$.  We claim that no vertex in
$G^{u,j}$ is colored by~$\phi$ with a color from $C^0$. In order to get a
contradiction, suppose that $\phi(v) =\alpha_i$ for some~$v$ in $G^{u,j}$ and
$i\in\left[\binom{p-1+t}{t}\right]$.  If $i=\binom{p-1+t}{t}$, then
$\phi(u) = \alpha_i = \phi(v)$, while $u$ and $v$ are adjacent in
$G_{(p,t,x,N)}$.  Since $p'\geq p\geq1$, the connected subgraph of
$G_{(p,t,x,N)}$ induced by $\set{u,v}$ contradicts the fact that $\phi$ is
$p'$-linear.  Now assume that $1\leq i< \binom{p-1+t}{t}$.  Let~$u_i$
and~$u_{i+1}$ be the ancestors of $u$ in $\RT^0$ of depth $i$ and $i+1$,
respectively; ($u_{i+1} = u$ is possible).  Since $X\geq2$, $u_i$ has some
child $u'$ distinct from $u_{i+1}$.  Let $u''$ be a leaf (vertex of depth
$\binom{p-1+t}{t}$) of $\RT^0$ in the subtree of $u'$.  Consider the paths
$u_i \cdots u_{\binom{p-1+t}{t}}$ and $w_i \cdots w_{\binom{p-1+t}{t}}$ in
$\RT$ where $u_{\binom{p-1+t}{t}}=u$, $w_i=u_i$, $w_{i+1}=u'$, and
$w_{\binom{p-1+t}{t}}=u''$.  Since~$\RT^0$ is
$\left(p,\binom{p-1+t}{t},X\right)$-brushed in $G^0$ there exist indices
$j_1,\ldots,j_k$ with $1\leq k \leq p-1$ and
$i=j_1 < j_2 < \cdots < j_k=\binom{p-1+t}{t}$ such that
$u_{j_1}\cdots u_{j_k}$ and $w_{j_1}\cdots w_{j_k}$ are paths in $G^0$
connecting $u_i = u_{j_1}$ with $u = u_{j_k}$ and $u_i = w_{j_1}$ with
$u'' = w_{j_k}$, respectively.  Consider the path
$P = u_{j_k} \cdots u_{j_1} w_{j_2} \cdots w_{j_k}$ connecting $u=u_{j_k}$ and
$u''=w_{j_k}$ in $G^0$ and recall that
$\phi(u_{j_{\ell}}) = \alpha_{j_{\ell}} =\phi(w_{j_{\ell}})$, for all
$\ell\in[k]$.  Thus, $\phi$ uses exactly $k$ colors on $P$ and the only unique
color on $P$ is $\phi(u_i) = \alpha_i$.  Now since $v$ and $u_{j_k}=u$ are
adjacent in $G_{(p,t,x,N)}$, the following sequence is a path in
$G_{(p,t,x,N)}$:
\[
v u_{j_k} \cdots u_{j_1} w_{j_2} \cdots w_{j_k}.
\]
Since $\phi(v) =\alpha_i = \phi(u_i)$, there is no color used by $\phi$ exactly once on that path. 
Furthermore, $\phi$ uses exactly $k \leq p-1 \leq p'$ colors on that path. 
This contradicts the fact that $\phi$ is $p'$-linear coloring of $G_{(p,t,x,N)}$.
Thus, no vertex in $G^{u,j}$ is colored by $\phi$ with a color from $C^0$

Let $u$ be a leaf of $\RT^0$ and $j\in [X]$. 
Since $G^{u,j}$ is isomorphic to $G_{(p,t-1,x,N)}$, we have that
$G^{u,j}$ has a $\left(p,M,x\right)$-brushed subgraph $\RT^{u,j}$ that satisifes \eqref{enu-invariant-x-ary-tree-linear-version}.
We fix a witnessing $x$-ary tree $\RT^{u,j}$ of depth $M$ and the respective sequence of colors $C^{u,j}=\left(\beta^{u,j}_1,\ldots,\beta^{u,j}_{M}\right)$ such that $\phi(v)=\beta^{u,j}_k$, 
for every $k\in\left[M\right]$ and every vertex $v$ of depth $k$ in $\RT^{u,j}$.
It follows from the previous paragraph that $\alpha_i \neq \beta^{u,j}_k$, for all $i\in\left[\binom{p-1+t}{t}\right]$ and $k\in\left[M\right]$.

Since $\phi$ takes less than $N$ values, there are less then $N^{M}$ possibilities for
the color sequence $C^{u,j}$.
Since $X = (x-1)N^{M}+1$, we find $x$ values $j\in [X]$ such that the
color sequences $C^{u,j}$ are identical, we let
$C^{u}=\left(\beta^u_1,\ldots,\beta^u_{M}\right)$ be this repeated color sequence.

We now perform a bottom-up traversal on  $\RT^{0}$.
We start by marking all the leaves as being \emph{visited}.
When reaching an internal vertex $v$ of $\RT^{0}$ its children
$u_1,\ldots,u_X$ have been visited and each of them has an
associated color sequence $C^{u_j}$. Again, since $X = (x-1)N^{M}+1$,
we find $x$ values $j\in [X]$ such that the
color sequences $C^{u_j}$ are identical, we define
$C^{v}$ to be this repeated color sequence and mark
$v$ as visited.

When all vertices have been maked visited every internal vertex $v$ in $\RT_0$
has $x$ children $u_1,\ldots,u_x$ such that $C^v=C^{u_i}$ for all $i\in[x]$.
This way starting from the root of $\RT^0$ we can filter out a subtree $\RT'$
of $\RT^0$ such that $\RT'$ is an $x$-ary tree of depth $\binom{p-1+t}{t}$
rooted at the root of $\RT_0$ with the properties that
\begin{enumerate}[label=\textup{(\roman*)},nosep]
\item for every vertex $v$ in $\RT'$, $v$ has the same depth in $\RT'$ as in $\RT^0$; and
\item there is a sequence $C=(\beta_1,\ldots,\beta_M)$ such that $C^v=C$ for all $v$ in $\RT'$.
\end{enumerate}

Finally, we define a tree $\RT$ as a tree obtained from $\RT'$ by attaching to
each leaf $u$ of~$\RT'$ exactly $x$ trees among $\RT^{u,j}$ for $j\in[X]$ such
that $C^{u,j}=C$.  Therefore $\RT$ is an $x$-ary tree.  Note also that the
depth of $\RT$ is the sum of the depth of $\RT'$ and $M=\binom{p+t-1}{t-1}$
(which is the depth of each $\RT^{u,j}$).  Thus, the depth of $\RT$ is
\[
\binom{p-1+t}{t} + \binom{p+t-1}{t-1} = \binom{p+t}{t}.
\]

We claim that $\RT$ is $\left(p,\binom{p+t}{t},x\right)$-brushed in $G_{(p,t,x,N)}$. 
To prove this, we need to verify the item~\eqref{def-brushed-ii} of the definition.
Let $v$ be a vertex in $\RT$ and $i$ be the depth of $v$ in $\RT$.

If $i> \binom{p-1+t}{t}$, then $v$ lies in one of the trees attached to $\RT'$
in the construction of $\RT$, say $v$ is in~$\RT^{u,j}$ for some leaf $u$ of
$\RT^0$ and $j\in[X]$.  Clearly, the depth of $v$ in $\RT^{u,j}$ is
$i'=i-\binom{p-1+t}{t}$.  Since $\RT^{u,j}$ is $(p,M,x)$-brushed in $G^{u,j}$,
we get the indices $j_1,\ldots,j_k$ with $1\leq k\leq p$ and
$i'=j_1 < \cdots < j_k = M$ such that for every leaf $w$ of $\RT^{u,j}$ that
lies in the subtree of $v$, if $u_{i'}\cdots u_M$ is a path in $\RT^{u,j}$
connecting $v=u_{i'}$ with $w=u_M$, then the sequence $u_{j_1}\cdots u_{j_k}$
induce a path in $G^{u,j}$.  Thus, in this case~\eqref{def-brushed-ii} holds
for $v$.

If $i\leq \binom{p-1+t}{t}$, then $v$ lies in $\RT'$.
Since $\RT^0$ is $(p-1,\binom{p-1+t}{t},X)$-brushed in $G^0$, we get the indices $j_1,\ldots,j_k$ with 
$1\leq k\leq p-1$ and $i=j_1 < \cdots < j_k = \binom{p-1+t}{t}$ such that for every leaf $w$ of $\RT^{0}$ that lies in the subtree of $v$, if $u_{i}\cdots u_{\binom{p-1+t}{t}}$ is a path in $\RT^0$ connecting $u=u_{i}$ with $w=u_{\binom{p-1+t}{t}}$, then the sequence $u_{j_1}\cdots u_{j_k}$ induces a path in $G^0$.
Consider a sequence $j_1,\ldots,j_k$ extended with $j_{k+1} = \binom{p+t}{t}$. 
Note that its length is $k+1 \leq p$.
Now consider any leaf $w'$ of $\RT$ that lies in the subtree of $v$. 
Let $w$ be the leaf of $\RT^0$ such that $w'$ lies in $\RT^{w,j}$ for some $j\in[X]$. 
Let $u_i\cdots u_{\binom{p+t}{t}}$ be the path in $\RT$ between $u_i=v$ and $u_{\binom{p+t}{t}}=w'$.
We claim that $u_{j_1}\cdots u_{j_{k+1}}$ is a path in $G_{(p,t,x,N)}$.
Indeed, $u_{j_1}\cdots u_{j_{k}}$ is a path in $G^0$ and $u_{j_k}=w$ is adjacent with $u_{j_{k+1}}=w'$ in $G^{u,j}$.
This completes the proof that $\RT$ is $\left(p,\binom{p+t}{t},x\right)$-brushed in $G_{(p,t,x,N)}$.

It remains to argue that the item~\eqref{enu-invariant-x-ary-tree-linear-version} holds for the coloring $\phi$ and the subgraph $\RT$ in $G_{(p,t,x,N)}$.
Consider the concatenated color sequence
\[
\left(\alpha_1,\ldots,\alpha_{\binom{p-1+t}{t}},\beta_1,\ldots,\beta_M\right).
\]
Let $i\in\left[\binom{p+t}{t}\right]$ and let $v$ be a vertex of depth $i$ in $\RT$.
If $i \leq \binom{p-1+t}{t}$, then $v$ is a vertex in $\RT^0$ and as such $\phi(v)=\alpha_i$. 
If $i > \binom{p-1+t}{t}$, then $v$ is a vertex in $\RT^{u,j}$ for some $u$ leaf of $\RT^0$ and $j\in[X]$. 
Since $\RT^{u,j}$ was chosen to attach to $\RT'$ in the construction of $\RT$, we have $C^{u,j}=C$.
Therefore, $\phi(v) = \beta_{i - \binom{p-1+t}{t}}$, as required.

\enlargethispage{3mm}
Therefore, $G_{(p,t,x,N)}$ satisfies both invariants.
This completes the proof.

\section{Lower bounds for graphs of bounded simple treewidth}
\label{sec-lower-bounds-subplanars}

In this section we prove Theorem~\ref{thm:simpletree}.\ref{enu-simpletree-low}.
We first deal with the case $k=2$, i.e., with the case of outerplanar graphs,
which was independently stated as Theorem~\ref{thm:pb+p-bounds}.\ref{enu-outerplanar-low}. 
We then generalize the proof for the case of larger $k$. The case $k=3$
yields Theorem~\ref{thm:pb+p-bounds}.\ref{enu-stackedtri-low}.

\subsection{Proof of the lower bound for outerplanar graphs}\hfill
\def\lll{^{(\ll)}}
\def\rrr{^{(\rr)}}
\def\ll{\ell}
\def\rr{r}
\def\sii{^{(\si)}}
\def\sib{^{(\sb)}}
\def\si{\sigma}
\def\sb{\bar{\sigma}}
\def\logs{\lfloor\log s\rfloor}
\def\logp{\lfloor\log p\rfloor}

\subsubsection*{1.  Constructing the family of graphs}\hfill

The \emph{tree of fans} $F(w,d)$ is obtained from a rooted complete
$w$-ary tree of depth~$d$ by connecting the children of each non-leaf 
vertex with a path. Note that $F(w,d)$ is outerplanar, the vertices
are partitioned into $d+1$ levels. Level 0 consists of the root vertex and
level $j$ has $w^j$ vertices, hence there are $\sum_{j=0}^d w^j$ vertices in total.

Let $s=\lfloor p/2\rfloor$ and $t = \lfloor s/2 \rfloor +1$ and
$f_2(p) = t\logs$.  We let $H$ be a tree of fans $F( s f_2(p), t )$ and
$\phi$ be a $p$-centered coloring of $H$. We claim that $\phi$ uses at
least $f_2(p)$ colors.

\subsubsection*{2.  A clean subgraph of $H$}\hfill

We identify a \emph{clean} subgraph $H^*$ of $H$ isomorphic to $F(s,t)$. 
The cleaning is done top-down. The root $v_0$ is
\emph{clean} by definition. If a vertex~$v$ is identified as clean we look at
the coloring of the path~$P_v$ on its children.  Let $U_v$ be the set of
unique colors, i.e., the set of colors which occur only once on~$P_v$.  If
$|U_v| \geq f_2(p)$ we have a proof that $\phi$ uses at least $f_2(p)$
colors. Otherwise considering the length of~$P_v$ which is $s f_2(p)$ and the
bound $|U_v|<f_2(p)$ we conclude that $P_v$ contains a subpath of size $s$
which contains no vertex with a unique color, i.e., every color that appears
on the subpath is assigned to at least two vertices of $P_v$. Fix such a
subpath $Q_v$ and declare its vertices to be \emph{clean}.  Then $H^*$ is
simply the subgraph of $H$ induced by the clean vertices.

\subsubsection*{3.  The spine for the master caterpillar of $H^*$}\hfill

A \emph{caterpillar} rooted at $w_0$ of \emph{depth} $d$ in $H$ consists of a
\emph{spine}, this is a path $w_0,w_1,\ldots w_d$ such that~$w_{i+1}$ is a
child of $w_i$, additionally there may be \emph{hairs}, they are leaves
attached to vertices of the spine.  We require that hairs attached to a spine
vertex $w$ are children of $w$, moreover, $w_d$ has no attached hairs, whence,
we also count it as a hair attached to $w_{d-1}$.  A caterpillar in $H$ is
\emph{$H^*$-based} if all spine vertices belong to $H^*$.

Starting from the root $v_0$ of $H^*$ we are going to identify a path
$S=v_0,v_1,\ldots,v_{t}$ in~$H^*$ such that $v_i$ is a vertex of level $i$.
Later we will see that there is a caterpillar $K_0$ in~$H$ with spine $S$ 
such that $\phi$ uses at least~$f_2(p)$ colors on $K_0$. 

Suppose the spine vertex $v_{\ell-1}$ has been identified. We now describe a
procedure to find~$v_\ell$. Let $Q$ be the path of clean children of $v_{\ell-1}$ in
$H^*$, this path consists of $s$ vertices. With iterated halving we identify a
sequence of nested intervals $Q=Q_0\supset Q_1 \supset\cdots$ which closes in
at a single vertex $v_\ell$.  The iteration requires at least $\logs$ steps.
When~$Q_j$ has been identified let $U_j$ be the set of unique colors on
$Q_j$. Since $|Q_j|\leq s\leq p$ and~$\phi$ is $p$-centered
$U_j\neq\emptyset$. Divide $Q_j$ into its left and right half $Q_j\lll$
and~$Q_j\rrr$, such that $||Q_j\lll| - |Q_j\rrr||\leq 1$. Let $U_j\lll$ and
$U_j\rrr$ be the subsets of those colors of $U_j$ which appear in $Q_j\lll$ and $Q_j\rrr$,
respectively. It may happen that one of $U_j\lll$ and $U_j\rrr$ is empty.

In the following we let $\si\in\{\ll,\rr\}$ and $\sb$ be such that
$\{\si,\sb\}=\{\ll,\rr\}$.  Let $d\sii$ be the minimum depth of an
$H^*$-based caterpillar rooted at a vertex of $Q_j\sii$ which contains
all the colors of $U_j\sib$.  If there is no such caterpillar we let
$d\sii=p$. 

\begin{klaim}\label{lem:h1+h2}
  $(d\lll-1)+(d\rrr-1) > p - |Q_j|$.
\end{klaim}
\begin{proof}
  For $\si\in\{\ll,\rr\}$ let $K\sii$ be an $H^*$-based caterpillar of depth $d\sii$
  rooted at a vertex of $Q_j\sii$ which collects all colors of
  $U_j\sib$.  
  We assume that every hair in $K\sii$ has a color from $U\sib$ as 
  otherwise  we simply remove it from $K\sii$.
  Let $\Gamma_0$ be the subgraph of $H$ which is obtained
  as the union of $Q_j$, $K\lll$, and~$K\rrr$.
  Note that the number of colors used by $\phi$ on $\Gamma_0$ is bounded 
  by $|Q_j| + (d\lll-1) + (d\rrr-1)$ as hairs, including the last vertices of the spines, 
  only reuse colors used on $Q_j$.

  In order to get a contradiction, suppose that $(d\lll-1)+(d\rrr-1)\leq p - |Q_j|$. 
  Then we have that the number of colors used by $\phi$ on $\Gamma_0$ is bounded by 
\[
|Q_j| + (d\lll-1) + (d\rrr-1) \leq |Q_j| + p - |Q_j| = p.
\]
  Next we are going to construct a connected supergraph $\Gamma$ of
  $\Gamma_0$ which contains the same colors as $\Gamma_0$, i.e., at
  most $p$, and has no unique color. Such a $\Gamma$ can not exist
  because $\phi$ is a $p$-centered coloring of $H$.

  Consider a color $\alpha$ that is used exactly once in $\Gamma_0$.
  Since all the colors used exactly once at $Q_j$ are repeated in $K\lll$ or $K\rrr$,
  the color $\alpha$ must appear in one of the $K\sii$, say~$K\lll$.
  Since all the colors of the hairs are also used at $Q_j$, the color $\alpha$
  must be used at a vertex $v$ of the spine of $K\lll$ but not within $Q_j$.
  Therefore the parent of $v$, say $v'$, also belongs to $K\lll$.
  Since $K\lll$ is $H^*$-based, $v$ and $v'$ belong to $H^*$, i.e., they are clean.
  Therefore there is a child $v''$ of $v'$ in $H$ with $v''\neq v$ and $\phi(v'')=\phi(v)=\alpha$.
  We add such a vertex $v''$ to $\Gamma$.
  We apply this procedure for every unique color $\alpha$ of $\Gamma_0$.
  
  The resulting graph $\Gamma$ is clearly connected and $|\phi(\Gamma)| = |\phi(\Gamma_0)|\leq p$ 
  while $\Gamma$ has no unique color under $\phi$.
  This contradiction completes the proof of the claim.
\end{proof}

{}From the claim we get a  $\si\in\{\ll,\rr\}$ with
$d\sii -1 \geq \frac{1}{2}(p-|Q_j|) \geq \frac{1}{2}(p - s) \geq \frac{p}{4} \geq t-1$.
Use this $\si$ to define $Q_{j+1}= Q_j\sii$ and $A_{j+1} = U_j\sib$.

The iterated halving ends with a vertex $v_\ell$ and a sequence
$A_1,A_2,\ldots,A_{\logs}$ of sets of colors. From the construction it follows that
$A_{i+1}$ is the set of unique colors of $Q_i$ which appear in
$Q_i\sib$ while for $j>i+1$ the colors of $A_j$ appear in
$Q_i\sii$. This shows that the sets $A_1,A_2,\ldots,A_{\logs}$ are pairwise
disjoint. Define $B_\ell=\bigcup_{i=1}^{\logs}A_i$. From the definition of
$d\sii$ and the inequality $d\sii\geq t$ we can deduce the following
observation, which will be crucial:
\begin{quote}
  Every $H^*$-based caterpillar of depth at most ${t-1}$ rooted at $v_\ell$
  misses at least one color from each $A_i$, i.e., it misses at least
  $\logs$ colors from $B_\ell$.
\end{quote}
\goodbreak

\subsubsection*{4.  Color collecting sub-caterpillars of the master}\hfill

Having defined the vertices $v_0,v_1,\ldots,v_{t}$ we let $K_\ell$ be the
caterpillar with spine $v_\ell,\ldots,v_{t}$ which includes all the children
of~$v_i$ in $H^*$ for each $i=\ell,\ldots,{t-1}$.  With~$C_\ell$ we denote the
set of colors of $K_\ell$.  Since for $0\leq j<\ell\leq t$ all vertices of
$K_\ell$ are also vertices of $K_j$ we get
$C_0\supset C_1\supset\cdots\supset C_t$.  Note that~$B_{\ell}$ is a set of
colors used by children of~$v_{\ell-1}$, therefore,
$B_\ell\subset C_{\ell-1}$.  For $\ell>0$, the caterpillar $K_\ell$ is an
$H^*$-based caterpillar of depth $t- \ell \leq {t-1}$ rooted at $v_\ell$,
therefore, $K_\ell$ misses at least $\logs$ colors from~$B_\ell$. Together
this shows that $|C_{\ell-1}| \geq |C_\ell|+\logs$, whence $|C_0|\geq t\logs$.

\subsection{Proof of the lower bound for simple treewidth k}\hfill
\label{sec:simple-k-low}

The proof of the lower bound for graphs of simple treewidth $k$ is based on the
same four steps as the proof of the lower bound for outerplanar graphs.

\subsubsection*{1.  Constructing the family of graphs}\hfill

The underlying graph for the lower bound is $G_k(w,d)$.

We give a recursive definition: $G_2(w,d) = F(w,d)$ where $F(w,d)$ is the tree
of fans, the vertices of level $d$ of this tree of fans constitute the
\emph{boundary} $L_2$ of $G_2(w,d)$.  Having defined $G_{j-1}(w,d)$ and its
boundary $L_{j-1}$ we aim at defining  $G_j(w,d)$.  We introduce an
additional parameter $\delta$ running from 1 to $d$ and graphs
$G_j(w,d,\delta)$. We then define $G_j(w,d)= G_j(w,d,d)$.

The graph $G_j(w,d,1)$ is obtained from $G_{j-1}(w,d)$ by making the root $v$
an universal vertex, i.e., for every vertex $u$ of the graph we ensure that the edge
$(u,v)$ belongs to the graph.  The boundary of $G_j(w,d,1)$ is the boundary
$L_{j-1}$ of $G_{j-1}(w,d)$.  From $G_j(w,d,\delta)$ we obtain
$G_j(w,d,\delta+1)$ by taking a copy $G(u)$ of $G_j(w,d,1)$ for each $u$ in
the boundary and identifying the root of $G(u)$ with $u$. The boundary of
$G_j(w,d,\delta+1)$ is the union of the boundaries of all the new subgraphs
$G(u)$.

The following properties of $G_k(w,d)$ will be useful.

\noindent (1) $G_k(w,d)$ contains a spanning subgraph isomorphic to
$F(w,d^{k-1})$.  Indeed, if we modify the definition slightly by deleting the
instruction \emph{make the root an universal vertex}, then througout the construction we only glue new
trees of fans to the boundary vertices of an already constructed tree of fans.
The boundary vertices are always the vertices of the largest level. This shows
that we obtain a tree of fans. For the parameters use induction and the fact that
the depth of the tree of fans underlying $G_j(w,d)$ is $d$ times
the depth of the tree of fans underlying $G_{j-1}(w,d)$.

\noindent (2) The simple treewidth of $G_j(w,d)$ is at most one more
than the simple treewidth of $G_{j-1}(w,d)$, whence  by induction $\stw(G_k(w,d)) \leq k$.

First note that the following argument implies $\stw(G_j(w,d,1)) \leq \stw(G_{j-1}(w,d))+1$.
If $\stw(G) = t$ and $G^+$ is obtained from $G$ by making some vertex
$v$ of $G$ an universal vertex,
then $\stw(G-v) \leq t$ and $\stw(G^+) = \stw(G-v) + 1$.

Since every 2-connected component of $G_j(w,d)$ is a copy of $G_j(w,d,1)$ we also
have $\stw(G_j(w,d)) \leq \stw(G_j(w,d,1))$.

\smallskip

Let $s=\lfloor \frac{p}{k} \rfloor$, and
$t =  \frac{k-1}{k}\frac{p}{2} $ and $h = \lfloor \frac{t}{k-1} \rfloor +1 =
\lfloor \frac{p}{2k} \rfloor +1$.  We also define 
$f_k(p) = h^{k-1} \logs$.

Let $H$ be a $G_k(s f_k(p), h)$ and $\phi$ be a $p$-centered coloring of
$H$. We claim that $\phi$ is using at least $f_k(p)$ colors.

\subsubsection*{2.  A clean subgraph of $H$}\hfill

In a first step we define a \emph{clean} subgraph $H^*$ of $H$. Let $F^*$ be
obtained by the cleaning procedure of the previous section applied to the tree
of fans $F(s f_k(p),h^{k-1})$ contained in $H$.  Note that $F^*$ is a
tree of fans $F(s,h^{k-1})$. Define $H^*$ as the graph
induced in $H$ by the clean vertices, i.e., by $V(F^*)$.

\subsubsection*{3.  The spine for the master caterpillar of $H^*$}\hfill

We adopt a notion of level in $H$ from the tree of fans
$F(s f_k(p),h^{k-1})$ contained in $H$. Let $\lev(u)$ denote the level of
vertex $u$, the level of the root is 0 and the level of vertices from the
boundary of $H$ is~$h^{k-1}$. If $vw$ is an edge of $H$ and
$\lev(v) < \lev(w)$ we say that $w$ is a
\emph{child} of $v$.

With this child relation the definition of a caterpillar and an $H^*$-based
caterpillar are exactly as in the previous subsection.

The next step is to define the spine of the master caterpillar.
This path $S=v_0,v_1,\ldots,v_{h^{k-1}}$ starting from the root $v_0$ of $H^*$
is computed within the clean spanning tree of fans $F^*$
contained in $H^*$.
The procedure is exactly as before, we only recall the main steps.

Suppose $v_{\ell-1}$ has been identified.
The next vertex $v_{\ell}$ of $S$ is determined by iterated halving of the path of
clean children of~$v_{\ell-1}$ in $F^*$. This produces $\logs$ sets of colors whose
union is denoted $B_\ell$.

With the same definitions as before we again have
$(d\lll-1)+(d\rrr-1) > p - |Q_j|$. 

Based on the inequality $s + 2t \leq p$ we get\footnote{The difference to the
  corresponding statement in the previous subsection comes from the fact that there the
  definition of $t$ had a $+1$ which is now shifted to the definition of
  $h$.}:
\begin{quote}
  Every $H^*$-based caterpillar of depth at most $t$ rooted at $v_\ell$
  misses at least $\logs$ colors from~$B_\ell$.
\end{quote}

\subsubsection*{4.  Color collecting sub-caterpillars of the master}\hfill

Having defined the vertices of the path $S$ we define a caterpillar $K_i$ for
each vertex~$v_i$ of $S$ as follows. The spine
$S_i =w_0,w_1,\ldots , w_{d_i}$ of $K_i$ is determined on the basis of
$S$. For the first vertex we take $w_0=v_i$. When $w_0 ,\ldots,w_j$ have
been determined and $w_{j}\neq v_{h^{k-1}}$, then let $w_{j+1}$ be the 
last child of $w_j$ in $S$, i.e, the last vertex $w$ on $S$ such
that $w_jw$ is an edge of $H^*$. With a vertex $w_j$ of the spine $S_i$
the caterpillar $K_i$ also includes all the children of $w_j$ in $H^*$
as hairs.

\begin{klaim}
  $V(K_{i+1}) \subset V(K_{i})$.
\end{klaim}

\begin{proof}
  {}From the definition $\lev(v_i) = \lev(v_{i+1})-1$.  First suppose that
  $\lev(v_i)$ is not divisible by $h$.  Then in the construction
  of $H$ the vertex $v_i$ was never element of a boundary, i.e., it wasn't
  used as an universal vertex of a
  subgraph. Therefore, all the children of $v_i$ in $H$ and in particular
  $v_{i+1}$ belong to level $\lev(v_i)+1$.  The spine $S_i$ thus starts with
  $v_i,v_{i+1}$ and by construction $S_i = v_iS_{i+1}$. This implies the claim
  in this case.

  Now suppose that $\ell\geq 1$ is maximal with the property that $h^\ell$
  divides $\lev(v_i)$. In this case $v_i$ was made the universal vertex
  of some copy
  $G(v_i)$ of $G_{\ell+1}(s f_k(p),h,1)$ used in the construction. If the
  spine $S_i$ thus starts with $v_i,w$, then $\lev(w) = \lev(v_i) + h^\ell$,
  i.e., $w = v_{i+h^\ell}$.  Moreover, all the elements $v_j$ of the spine $S$
  which are between $v_i$ and $w$ and all their children also belong to
  $G(v_i)$ and hence as children of $v_i$ also belong to $K_i$.  The previous
  considerations also imply that $w$ belongs to the spine of $S_{i+1}$.  This
  proves the claim in this case.
\end{proof}

Let $C_i$ be the set of colors of $K_i$. From the previous claim
we get the nesting of the colorsets
$C_0\supset C_1 \supset \ldots \supset C_{h^{k-1}}$.

{}From the construction we know that for each vertex $v_i$ in $S$ the caterpillar $K_i$ is an $H^*$-based
caterpillar.
\begin{klaim}
  For $i>0$, the depth $d_i$ of $K_{i}$ is smaller than $(k-1)h$.
\end{klaim}

\begin{proof}
  Let $S_i =w_0,w_1,\ldots w_{d_i}$. Clearly $w_{d_i}=v_{h^{k-1}}$.  Let
  $j$ be the least index such that $\lev(w_j)$ is divisible by $h^{k-2}$.
  Then $\lev(w_\ell)$ is also divisible by $h^{k-2}$ for all $j < \ell \leq d_i$.
  Therefore, there are at most $h$ vertices~$w_\ell$ with $\lev(w_\ell)$ divisible
  by $h^{k-2}$ in $S_i$.

  {}From the minimality of $j$ it follows that $w_0,\ldots w_{j}$
  are all taken from the same copy of $G_{k-1}(s f_k(p),h)$ used in the
  construction of $H$. By induction, the base case is a tree of fans,
  we can asume that $j < (k-2)h$. 
\end{proof}

Since $(k-1)h \leq t+1$
we conclude that the depth of $K_{i}$ is at most $t$ for all
$i > 0$. 

The set $B_i$ is a set of colors used by children of the predecessor $v_{i-1}$
of $v_i$, therefore, $B_i \subset C_{i-1}$. Caterpillar $K_i$ is $H^*$-based,
rooted at $v_i$, and of depth at most $t$, therefore, $K_i$ misses at
least $\logs$ colors from $B_i$. Together this shows that
$|C_{i-1}| \geq |C_i|+\logs$, whence $|C_0|\geq h^{k-1}\logs$.  This
completes the proof that $\phi$ uses at least $f_k(p)$ colors
on $H$.

\section{Further directions}
We finish the paper with our two favourite problems.
\begin{conjecture}
Planar graphs admit $p$-centered colorings with $\Oh(p^2\log p)$ colors.
\end{conjecture}
\begin{question}
Do outerplanar graphs admit $p$-linear colorings with $\Oh(p)$ colors? 
\label{que:outerplanar}
\end{question}
The best we know for Question~\ref{que:outerplanar} is $\Oh(p\log p)$ by Theorem~\ref{thm:pb+p-bounds}.\ref{enu-outerplanar-up}.

Another line of thought is that our upper bound for bounded degree graphs implies the existence of $p$-centered colorings 
with $\Oh(p)$ colors of planar grids.
Interestingly, the authors have not been able to provide a direct construction for such a coloring.

\section*{Acknowledgements}
We are grateful to Micha{\l} Pilipczuk for pointing out 
that the upper bound for bounded degree graphs was the only missing 
piece to get a polynomial bound for graphs excluding a fixed topological minor.
We thank Gwenaël Joret for pointing out how to set up an entropy compression argument 
for Theorem~\ref{thm:bounded-degree} to get
an upper bound $\Oh(p\cdot\Delta^{2-\epsi})$ with $\epsi=\Omega(1/p)$. 
In the earlier version of this paper~\cite{DFMS20} we have just $\epsi>0$.
We are also grateful to Hendrik Schrezenmaier, Micha{\l} Seweryn, and Raphael Steiner
for many lively discussions on $p$-centered colorings.

\bibliographystyle{amsplain}
\bibliography{sgt}

\begin{aicauthors}
\begin{authorinfo}[MD]
Michał Dębski\\
Faculty of Informatics\\
Masaryk University\\
Brno, Czech Republic\\
and\\
Faculty of Mathematics and Information Sciences\\
Warsaw University of Technology\\
Warsaw, Poland\\
\email{michal.debski87@gmail.com}
\end{authorinfo}
\begin{authorinfo}[SF]
Stefan Felsner\\
Institut f\"ur Mathematik\\
Technische Universit\"at Berlin\\
Berlin, Germany\\
\email{felsner@math.tu-berlin.de}
\end{authorinfo}
\begin{authorinfo}[PM]
Piotr Micek\\
Institute of Theoretical Computer Science\\
Faculty of Mathematics and Computer Science\\
Jagiellonian University\\
Krak\'ow, Poland\\
\email{piotr.micek@uj.edu.pl}
\end{authorinfo}
\begin{authorinfo}[FS]
Felix Schröder\\
Institut f\"ur Mathematik\\
Technische Universit\"at Berlin\\
Berlin, Germany\\
\email{fschroed@math.tu-berlin.de}
\end{authorinfo}
\end{aicauthors}


\end{document}